\documentclass[12pt,a4paper]{amsart}
\calclayout
\numberwithin{equation}{section}
\allowdisplaybreaks
\theoremstyle{plain}

\newtheorem{thrm}{Theorem}[section]
\newtheorem{lemma}[thrm]{Lemma}
\newtheorem{prop}[thrm]{Proposition}

\newtheorem{remark}{Remark}

\numberwithin{equation}{section}

\usepackage{enumerate}

\usepackage{amssymb}
\usepackage{mathtools}
\mathtoolsset{showonlyrefs}
\usepackage{mathrsfs}
\usepackage{comment}
\usepackage{hyperref}
\usepackage{xcolor}
\usepackage{graphicx}
	\def\MR#1{\href{https://urldefense.com/v3/__http://www.ams.org/mathscinet-getitem?mr=*1*7D*7BMR-*1__;IyUlIw!!DZ3fjg!oHzqzx-fk_p2DpzqdMBj54OAPH9BYzIUMFry2bFVial4aXytl4As9A_7US17LHEg$ }}
	\def\ARXIV#1{\href{https://urldefense.com/v3/__https://arxiv.org/abs/*1*7D*7BarXiv:*1__;IyUlIw!!DZ3fjg!oHzqzx-fk_p2DpzqdMBj54OAPH9BYzIUMFry2bFVial4aXytl4As9A_7UeWlCDxg$ }}
	
\def\DOI#1{\href{https://urldefense.com/v3/__https://doi.org/*1*7D*7Bdoi:*1__;IyUlIw!!DZ3fjg!oHzqzx-fk_p2DpzqdMBj54OAPH9BYzIUMFry2bFVial4aXytl4As9A_7UVT3CNkl }}

\def\P{\mathbb{P} }

\def\R{\mathbb{R} }
\def\N{\mathbb{N} }

\def\E{\mathbb{E} }

\def\F{\mathcal{F}}
\def\H{\mathcal{H}}

\begin{document}
\title[Branching Brownian motion in a periodic environment]{Branching Brownian motion in a periodic environment and
uniqueness of pulsating travelling waves}
\thanks{The research of this project is supported by the National Key R\&D Program of China (No. 2020YFA0712900).}
\author[Y.-X. Ren, R. Song and F. Yang]{Yan-Xia Ren, Renming Song and Fan Yang}
\address{Yan-Xia Ren\\ LMAM School of Mathematical Sciences \& Center for
Statistical Science\\ Peking University\\ Beijing 100871\\ P. R. China}
\email{yxren@math.pku.edu.cn}
\thanks{The research of Y.-X. Ren is supported in part by NSFC (Grant Nos. 12071011  and 11731009) and LMEQF.
}
\address{Renming Song\\ Department of Mathematics\\ University of Illinois at Urbana-Champaign \\ Urbana \\ IL 61801\\ USA}
\email{rsong@illinois.edu}
\thanks{The research of R. Song is supported in part by a grant from the Simons Foundation (\#429343, Renming Song)}
\address{Fan Yang\\ School of Mathematical Sciences \\ Peking University\\ Beijing 100871\\ P. R. China}
\email{fan-yang@pku.edu.cn}
\begin{abstract}
Using one-dimensional branching Brownian motion in a periodic environment, we give probabilistic proofs of the asymptotics and uniqueness of  pulsating travelling waves of the F-KPP equation in a periodic environment.
This paper is a sequel to [Ren et al. Branching Brownian motion in a periodic environment and
existence of pulsating travelling waves], in which we proved
the existence of the pulsating travelling waves in the supercritical and critical cases using the limits of the additive and derivative martingales of  branching Brownian motion in a periodic environment.
\end{abstract}
\maketitle



\section{Introduction}
 McKean \cite{Mc} established the connection between  branching Brownian motion (BBM) and
the Fisher-Kolmogorov-Petrovskii-Piskounov (F-KPP)
reaction-diffusion equation
\begin{equation}\label{KPPeq1}
\frac{\partial \mathbf u}{\partial t} = \frac{1}{2} \frac{\partial^2 \mathbf u}{\partial x^2} + \beta({\mathbf f}({\mathbf u})-{\mathbf u}),
\end{equation}
where $\mathbf f$ is the generating function of the offspring distribution and $\beta$ is the (constant) branching rate of BBM.
The F-KPP equation has been studied intensively by both analytic techniques (see, for example, Kolmogorov et al. \cite{KPP} and Fisher \cite{Fish}) and probabilistic methods
(see, for instance, McKean \cite{Mc}, Bramson \cite{Bramson78,Bramson83}, Harris \cite{Harris99} and Kyprianou \cite{Ky}).

A travelling wave solution of \eqref{KPPeq1} with speed $c$ is a solution of the following equation:
\begin{equation}\label{travel}
\frac{1}{2} \Phi_c'' + c\Phi_c' +
\beta({\mathbf f}(\Phi_c)-\Phi_c) = 0.
\end{equation}
If $\Phi_c$ is a solution of \eqref{travel}, then ${\mathbf u}(t,x) = \Phi_c(x-ct)$ satisfies \eqref{KPPeq1}.
Using the relation between the F-KPP equation \eqref{KPPeq1} and BBM,
Kyprianou \cite{Ky} gave probabilistic proofs of the existence, asymptotics and uniqueness of travelling wave solutions.
In this paper, we study the following more general F-KPP equation in which the constant $\beta$ is replaced by a continuous and 1-periodic function $\mathbf g$:
\begin{equation} \label{KPPeq2}
\frac{\partial \mathbf u}{\partial t} = \frac{1}{2} \frac{\partial^2 \mathbf u}{\partial x^2} +
\mathbf g\cdot
(\mathbf f(\mathbf u)-\mathbf u),
\end{equation}
where ${\mathbf u}:\R^+ \times \R \rightarrow [0,1]$.
In \cite{RSYa}, we have shown that the above equation is related to branching Brownian motion in a periodic environment.
	
Now we describe the branching Brownian motion in a periodic environment.
Initially there is a single particle $v$ at the origin of the real line. This particle moves as a standard Brownian motion $B = \{B(t), t\geq 0 \}$ and produces a random number of offspring,
$1+L$, after a random time $\eta_v$. We assume that $L$ has distribution $\{p_k, k\geq 0\}$ with $m:=\sum_{k\geq 0} kp_k\in(0,\infty)$. Let $b_v$ and $d_v$ be the birth time and death time of the particle $v$, respectively, and $X_v(s)$ be the location of the particle $v$ at time $s$, then $\eta_v = d_v-b_v$, the lifetime of $v$, satisfies
\begin{equation}\label{KPP-P}
\P_x
\left(\eta_v>t \;|\; b_v, \{X_v(s): s\geq b_v \} \right) =
\exp\left\{ -\int_{b_v}^{b_v+t} \mathbf g(X_v(s)) \mathrm{d}s \right\},
\end{equation}
where we assume the branching rate function $\mathbf g\in C^1(\R)$ is strictly positive and 1-periodic.
Starting from their points of creation, each of these children evolves independently.

Let $N_t$ be the set of particles alive at time $t$ and $X_u(s)$ be the position of the particle $u$ or its ancestor at time $s$ for any $u\in N_t$, $s\leq t$. Define
$$Z_t = \sum_{u\in N_t} \delta_{X_u(t)},$$
and $\F_t = \sigma(Z_s: s\leq t)$.
$\{Z_t: t\ge 0\}$ is called a branching Brownian motion in a periodic environment
(BBMPE).
Let $\P_x$ be the law of $\{Z_t: t\ge 0\}$ when the initial particle starts at $x\in\R$, that is $\P_x(Z_0=\delta_x) = 1$ and $\E_x$ be expectation with respect to $\P_x$.
For simplicity, $\P_0$ and $\E_0$ will be written as $\P$ and $\E$, respectively. Notice that the distribution of $L$ does not depend on the spatial location.
In the remainder of this paper, expectations with respect to $L$ will be written as
$\mathbf E$. The notation in this paper is the same as those in \cite{RSYa}.

As stated in \cite{RSYa}, the F-KPP equation related to BBMPE
is given by  \eqref{KPPeq2} with ${\mathbf f}(s)={\mathbf E}(s^{L+1})$.
Travelling wave solutions, that is solutions satisfying \eqref{travel}, do not exist. However, we can consider so-called pulsating travelling waves, that is, solutions $\mathbf u:\R^+\times\R\rightarrow [0,1]$ to \eqref{KPPeq2} satisfying
\begin{equation}\label{pul_travel}
\mathbf u(t+\frac{1}{\nu}, x) = \mathbf u(t,x-1),
\end{equation}
as well as the boundary condition
\begin{equation}
\lim_{x\rightarrow-\infty} \mathbf u(t,x) = 0,\quad \lim_{x\rightarrow+\infty} \mathbf u(t,x) =1,
\end{equation}
when $\nu>0$, and
\begin{equation}
\lim_{x\rightarrow-\infty} \mathbf u(t,x) = 1,\quad \lim_{x\rightarrow+\infty} \mathbf u(t,x) =0,
\end{equation}
when $\nu<0$. $\nu$ is called the wave speed.
It is known that there is a constant $\nu^*>0$ (defined below) such that
when $|\nu|<\nu^*$ (called subcritical case) no such solution exists, whereas for each $|\nu|\geq \nu^*$ (where $|\nu|>\nu^*$ is called the supercritical case and $|\nu|=\nu^*$ is called the critical case)
there exists a unique, up to time-shift, pulsating travelling wave (see Hamel et al. \cite{HNRR}).
	
In \cite{RSYa}, we studied the limits of the additive and derivative martingales of BBMPE, and by using these limits, gave a probabilistic proof for the existence of pulsating travelling waves.
In this paper, using the relation between BBMPE and related F-KPP equation, we give probabilistic proofs of the asymptotics and uniqueness of pulsating travelling waves.
These extend the results of Kyprianou \cite{Ky} for classical BBM to BBMPE.
However, the methods in  Kyprianou \cite{Ky} do not work for BBMPE.
We will adapt the ideas from \cite{Harris99}. The non-homogeneous nature of the environment makes the actual arguments much more delicate.

Before stating our main results, we first introduce the minimal speed $\nu^*$.
For every $\lambda\in\R$, let $\gamma(\lambda)$ and $\psi(\cdot,\lambda)$ be the principal eigenvalue and the corresponding positive eigenfunction of the periodic problem: for all $x\in \R$,
\begin{equation}\label{eigen}
\begin{split}
\frac{1}{2} \psi_{xx}(x, \lambda) - \lambda \psi_x(x, \lambda) + (\frac{1}{2}\lambda^2
+m\mathbf g(x))\psi(x, \lambda) &= \gamma(\lambda)\psi(x, \lambda),\\
\psi(x+1,\lambda) &= \psi(x,\lambda).
\end{split}
\end{equation}
We normalize $\psi(\cdot,\lambda)$ such that $\int_0^1 \psi(x,\lambda) dx = 1$. Define
\begin{equation}\label{crit_v}
\nu^*:=\min_{\lambda>0} \frac{\gamma(\lambda)}{\lambda},
\quad \lambda^*:= \underset{\lambda>0}{\arg\min} \frac{\gamma(\lambda)}{\lambda}.
\end{equation}
$\nu^*$ is  the minimal wave speed (see \cite{HNRR}) and the existence of $\lambda^*$ is proved in \cite{LTZ}.

Using the property $\mathbf u\left(t+\frac{1}{\nu},x\right) = \mathbf u(t,x-1)$, we can define $\mathbf u(-t,x)$ for any $t>0$. To be more specific, let $\lceil x \rceil$ be the smallest integer greater than or equal to $x$
and $\lfloor x \rfloor$ be the integral part of $x$.
When $\nu>0$, define
\begin{equation}
\mathbf u(-t,x) = \mathbf u\left(-t+\frac{\lceil \nu t \rceil}{\nu}, x+\lceil \nu t \rceil\right),\quad t>0, x\in\R.
\end{equation}
When $\nu<0$, define
\begin{equation}
\mathbf u(-t,x) = \mathbf u\left(-t+\frac{\lfloor \nu t \rfloor}{\nu}, x+\lfloor \nu t \rfloor\right),\quad t>0, x\in\R.
\end{equation}
Then $\mathbf u(t,x)$  satisfies the F-KPP equation \eqref{KPPeq2} and \eqref{pul_travel} in $\R\times\R$.

Our first two main results give the asymptotic behaviors of pulsating travelling waves in the supercritical case of $|\nu|>\nu^*$ and the critical case of $|\nu| = \nu^*$.
\begin{thrm}\label{thrm_asym_super}
Suppose  $\mathbf u(t,x)$ is
a pulsating travelling wave with speed $\nu> \nu^*$ and
$\lambda\in (0,\lambda^*)$ satisfies
$\nu=\frac{\gamma(\lambda)}{\lambda}$. If $\mathbf E(L\log^+L)<+\infty$, then there exists $\beta>0$ such that
	\begin{equation}
	1-\mathbf u\left(\frac{y-x}{\nu}, y\right) \sim \beta e^{-\lambda x}\psi(y,\lambda) \text{ as } x\rightarrow +\infty \text{ uniformly in } y\in [0,1].
	\end{equation}
\end{thrm}

\begin{thrm}\label{thrm_asym_crit}
Suppose $\mathbf u(t,x)$ is
a pulsating travelling wave with speed $\nu= \nu^*$. If $\mathbf E(L(\log^+L)^2)<\infty$, then there exists $\beta>0$ such that
	\begin{equation}
	1-\mathbf u\left(\frac{y-x}{\nu^*}, y\right) \sim \beta xe^{-\lambda^* x}\psi(y,\lambda^*) \text{ as } x\rightarrow +\infty \text{ uniformly in } y\in [0,1].
	\end{equation}
\end{thrm}

\begin{remark}
	By symmetry, we also have the asymptotic behaviors of pulsating travelling waves with negative speed.
	In the supercritical case of $\nu < -\nu^*$, suppose $\mathbf u(t,x)$ is
	a pulsating travelling wave with speed
	$\nu$ and
	$\lambda\in (-\lambda^*,0)$ satisfies
	$\nu=\frac{\gamma(\lambda)}{\lambda}$. If $\mathbf E(L\log^+L)<+\infty$, then there exists $\beta>0$ such that
	\begin{equation}
	1-\mathbf u\left(\frac{y-x}{\nu}, y\right) \sim \beta e^{-\lambda x}\psi(y,\lambda) \text{ as } x\rightarrow -\infty \text{ uniformly in } y\in [0,1].
	\end{equation}
	In the critical case, suppose $\mathbf u(t,x)$ is
	a pulsating travelling wave with speed $-\nu^*$. If $\mathbf E(L(\log^+L)^2)<\infty$, then there exists $\beta>0$ such that
	\begin{equation}
	1-\mathbf u\left(\frac{y-x}{-\nu^*}, y\right) \sim \beta |x|e^{\lambda^* x}\psi(y,-\lambda^*) \text{ as } x\rightarrow -\infty \text{ uniformly in } y\in [0,1].
	\end{equation}
\end{remark}

For any $\lambda\in\R$, define
\begin{equation}\label{mart_add}
W_t(\lambda) = e^{-\gamma(\lambda)t} \sum_{u\in N_t} e^{-\lambda X_u(t)} \psi(X_u(t),\lambda),
\end{equation}	
and
\begin{equation}\label{mart_deriv}
\partial W_t(\lambda) :=
e^{-\gamma(\lambda)t} \sum_{u\in N_t} e^{-\lambda X_u(t)} \bigg{(} \psi(X_u(t),\lambda) (\gamma'(\lambda)t+X_u(t)) - \psi_{\lambda}(X_u(t),\lambda)  \bigg{)}.
\end{equation}

It follows from \cite[Theorem 1.1]{RSYa} that,
for any $\lambda\in\R$ and $x\in\R$, $\{(W_t(\lambda))_{t\geq 0}, \P_x\}$ is a martingale, and called the  additive martingale.
The limit $W(\lambda,x) := \lim_{t\uparrow\infty} W_t(\lambda)$ exists $\P_x$-almost surely.
Moreover, $W(\lambda,x)$ is an $L^1(\P_x)$-limit when $|\lambda|>\lambda^*$ and $\mathbf E(L\log^+L) < \infty$;
and $W(\lambda,x) = 0$ $\P_x$-almost surely when $|\lambda|\le\lambda^*$ or  $|\lambda|>\lambda^*$ and $\mathbf E(L\log^+L) =\infty$.
It follows from \cite[Theorem 1.2]{RSYa} that,
for any $\lambda\in\R$ and $x\in\R$,
$\{(\partial W_t(\lambda))_{t\geq 0}, \P_x\}$ is a martingale, and called the derivative martingale.  For all $|\lambda|\geq \lambda^*$, the limit $\partial W(\lambda,x) := \lim_{t\uparrow\infty} \partial W_t(\lambda)$ exists $\P_x$-almost surely.
Moreover, if  $\mathbf E(L(\log^+L)^2)<\infty$, $\partial W(\lambda,x)\in(0,\infty)$ when $\lambda=\lambda^*$, and  $\partial W(\lambda,x)\in(-\infty,0)$ when $\lambda=-\lambda^*$.   If  $|\lambda|> \lambda^*$ or $|\lambda|=\lambda^*$ and  $\mathbf E(L(\log^+L)^2)=\infty$, $\partial W(\lambda,x) = 0$ $\P_x$-almost surely.

Using Theorem \ref{thrm_asym_super}, Theorem \ref{thrm_asym_crit} and \cite[Theorem 1.3]{RSYa}, we can prove the following result, which gives the existence and uniqueness of pulsating travelling waves.

\begin{thrm}\label{thrm3}
$\mathrm{(i)}$  \textbf{Supercriticality case.}
If $|\nu|>\nu^*$ and $\mathbf E(L\log^+L) < \infty$,
then there is a unique, up to time-shift, pulsating travelling wave with speed $\nu$  given by
\begin{equation}
\mathbf u(t,x)
= \E_x \left(\exp\left\{ -e^{\gamma(\lambda)t} W(\lambda,x) \right\} \right),
\end{equation}
where $|\lambda| \in (0,\lambda^*)$ is such that
$\nu = \frac{\gamma(\lambda)}{\lambda}$.

$\mathrm{(ii)}$ \textbf{Criticality case.}
If $|\nu|=\nu^*$ and $\mathbf E(L(\log^+L)^{2}) < \infty$, then there is a unique,
up to time-shift, pulsating travelling wave with speed $\nu$ given by
\begin{equation}
 \mathbf u(t,x) = \E_x\left(\exp\left\{ -e^{\gamma(\lambda)t} \partial W(\lambda,x) \right\} \right),
\end{equation}
where $\lambda = \lambda^*$  if $\nu = \nu^*$, and $\lambda = -\lambda^*$ if $\nu = -\nu^*$.
\end{thrm}

\section{Preliminaries}

\subsection{Properties of principal eigenvalue and eigenfunction}
In this section, we recall some properties of $\gamma(\lambda)$ and $\psi(x,\lambda)$ from \cite{RSYa}.
By \cite[Lemma 2.1]{RSYa}, the function $\gamma$ is analytic, strictly convex and  even on $\R$. There exists a unique $\lambda^*>0$ such that
	\begin{equation*}
	\nu^*=\frac{\gamma(\lambda^*)}{\lambda^*} = \min_{\lambda>0} \frac{\gamma(\lambda)}{\lambda}>0.
	\end{equation*}
	Furthermore
	\begin{equation}\label{gamma'}
	\lim_{\lambda\rightarrow -\infty}\gamma'(\lambda) = -\infty, \quad \lim_{\lambda\rightarrow +\infty}\gamma'(\lambda) = +\infty.
	\end{equation}	
By \cite[Lemma 2.2]{RSYa},
we have $\gamma'(\lambda^*) = \dfrac{\gamma(\lambda^*)}{\lambda^*}$,
\begin{equation}\label{lemma_com(2)}
\gamma'(\lambda) < \dfrac{\gamma(\lambda)}{\lambda},
 \quad \mbox{ on } (0, \lambda^*) \quad \mbox{ and }
\quad
\gamma'(\lambda) > \dfrac{\gamma(\lambda)}{\lambda},\quad \mbox{ on }
(\lambda^*, \infty).
\end{equation}
By \cite[Lemma 2.5]{RSYa}, we have $\psi(x,\cdot)\in C(\R)\cap C^1(\R\setminus\{0\})$ and $\psi_{\lambda}(x,\lambda)$ satisfies
\begin{align}\label{eigen_diff}
&\frac{1}{2} \psi_{\lambda xx}(x, \lambda) - \psi_x(x, \lambda) - \lambda \psi_{\lambda x}(x, \lambda) + (\frac{1}{2}\lambda^2 + m\mathbf g(x))\psi_{\lambda}(x, \lambda) + \lambda\psi (x, \lambda)\\
&= \gamma(\lambda) \psi_{\lambda} (x, \lambda) + \gamma'(\lambda) \psi(x, \lambda).\nonumber
\end{align}
Define
\begin{equation}\label{def_phi}
\phi(x,\lambda) := e^{-\lambda x} \psi(x,\lambda), \quad x\in\R.
\end{equation}
Then $\phi(x,\lambda)$ satisfies
\begin{equation}\label{phi-equ}
\frac{1}{2} \phi_{xx}(x,\lambda) + m\mathbf g(x) \phi(x,\lambda) = \gamma(\lambda) \phi(x,\lambda),
\end{equation}
and $\phi_{\lambda}(x,\lambda)$ satisfies
\begin{equation}\label{phi_lambda_eq}
\frac{1}{2} \phi_{\lambda xx}(x,\lambda) + m\mathbf g(x) \phi_{\lambda}(x,\lambda) = \gamma'(\lambda)\phi(x,\lambda) + \gamma(\lambda)\phi_{\lambda}(x,\lambda).
\end{equation}
Define
\begin{equation}\label{def-h}
h(x):= x - \frac{\psi_{\lambda}(x,\lambda)}{\psi(x,\lambda)},
\end{equation}
we also have
\begin{equation}\label{def-h2}
h(x)=-\frac{\phi_{\lambda}(x,\lambda)}{\phi(x,\lambda)}.
\end{equation}
It is easy to see that $h'$ is 1-periodic and continuous and \cite[Lemma 2.10]{RSYa} shows $h'$ is strictly positive.

\subsection{Measure change for Brownian motion}\label{ss:mcBB}
Martingale change of measures for Brownian motion will play important roles in our arguments.
In this section, we state the results of \cite{RSYa} about martingale change of measures.

Define
\begin{equation}\label{mart_eta-t}
\Xi_t(\lambda):= e^{-\gamma(\lambda)t - \lambda B_t + m\int_0^t\mathbf g(B_s)\mathrm{d}x} \psi(B_t,\lambda),
\end{equation}
then by \cite[Lemma 2.6]{RSYa}, $\{\Xi_t(\lambda), t\geq 0\}$  is a $\Pi_x$-martingale. Define a probability measure $\Pi^\lambda_x$ by
\begin{equation}\label{meas_Plambda}
\frac{\mathrm{d}\Pi_x^{\lambda}}{\mathrm{d}\Pi_x}\bigg{|} _{\mathcal{F}_t^B} =
\frac{\Xi_t(\lambda)}{\Xi_0(\lambda)},
\end{equation}
where $\{\mathcal{F}_t^B: t\geq 0\}$
is the natural filtration of Brownian motion.
We have shown in \cite{RSYa} that
$\{B_t, \Pi_x^{\lambda}\}$ is a diffusion with infinitesimal generator
\begin{equation}\label{Y_infin}
(\mathcal{A}f)(x) = \frac{1}{2} \frac{\partial^2 f(x)}{\partial x^2} + \left(\frac{\psi_x(x,\lambda)}{\psi(x,\lambda)}-\lambda\right)\frac{\partial f(x)}{\partial x}.	
\end{equation}
{\it In the remainder  of this paper, we always assume
that $\{Y_t, t\geq 0; \Pi_x^{\lambda}\}$ is a diffusion with infinitesimal generator \eqref{Y_infin}.}
It follows from \cite[Lemma 2.8]{RSYa} that, for any $x\in\R$,
\begin{equation}\label{lemma_slln}
\frac{Y_t}{t} \rightarrow -\gamma'(\lambda),\quad \Pi_x^{\lambda}\mbox{-a.s.}
\end{equation}

Define
\begin{equation}
M_t: = \gamma'(\lambda)t + h(Y_t) -h(Y_0), \quad t\ge 0.
\end{equation}
By  \cite[Lemma 2.12]{RSYa},   $\{M_t, t\geq 0; \Pi^\lambda_x\}$ is a martingale.
Moreover, there exist two constants $c_2>c_1>0$
such that the quadratic variation $\langle M \rangle_t$ satisfying
\begin{equation}\label{quadratic}
\langle M \rangle_t = \int_0^t \left( h'(Y_s) \right)^2 \mathrm{d}s \in [c_1t, c_2t].
\end{equation}

For any $x\in\R$, define an $\{\F_t^B \}$ stopping time
\begin{equation}\label{def_tauB}
\tau^x_{\lambda}:= \inf\left\{t\geq 0: \; h(B_t) \leq - x - \gamma'(\lambda) t  \right\}.
\end{equation}
Define
\begin{equation}\label{mart_Lambda}
\Lambda_t^{(x,\lambda)}:= e^{-\gamma(\lambda)t - \lambda B_t + m\int_0^t\mathbf g(B_s)\mathrm{d}s} \psi(B_t,\lambda)
\left( x+\gamma'(\lambda)t + h(B_t) \right)
\textbf{1}_{\{\tau^x_{\lambda} > t \}},
\end{equation}
then \cite[Lemma 2.11]{RSYa} shows that for any $x, y\in \R$ with $y>h^{-1}(-x)$, $\{\Lambda_t^{(x,\lambda)},t\geq 0\}$ is a $\Pi_y$-martingale. For $x, y\in \R$ with $y>h^{-1}(-x)$, define a new probability measure $\Pi^{(x, \lambda)}_y$ by
\begin{equation}\label{meas_change}
\frac{\mathrm{d}\Pi_y^{(x,\lambda)}}{\mathrm{d}\Pi_y}\bigg{|} _{\mathcal{F}_t^B} = \frac{\Lambda_t^{(x,\lambda)}}{\Lambda_0^{(x,\lambda)}}.
\end{equation}
By \cite[Section 2.2]{RSYa}, if $\{B_t, t\geq 0; \Pi_y \}$ is a standard Brownian motion starting at $y$, then $\left\{x +h(y)+ M_{T(t)}, t\geq 0; \Pi_y^{(x,\lambda)}\right\}$
is a standard Bessel-3 process
starting at $x +h(y)$, where $M_t = \gamma'(\lambda)t + h(B_t) -h(B_0)$ and $T(s) = \inf\left\{t>0:\langle M \rangle_t>s \right\} = \inf\left\{t>0: \int_0^t \left( h'(B_s) \right)^2 \mathrm{d}s>s \right\}$.

\section{Proof of Theorem \ref{thrm_asym_super}}
\begin{proof}[Proof of Theorem \ref{thrm_asym_super}]
	We fix $\nu> \nu^*$  in this proof and so $\lambda$ is fixed also.
	We will prove the theorem in five steps. In the first four steps, we assume the number of offspring is 2, that is $L=1$. In the last step, we prove the result for general $L$.
	
	\textbf{Step 1} Suppose $L=1$ and thus $m=1$. Let $\mathbf w(t,x) = 1 -\mathbf u(t,x)$, then $\mathbf w(t,x)$ satisfies
	\begin{equation}\label{KPPequation_w}
    \begin{cases}
	&\frac{\partial \mathbf w}{\partial t} = \frac{1}{2} \frac{\partial^2 \mathbf w}{\partial x^2} +
	\mathbf g\cdot
	(\mathbf w-\mathbf w^2),\\
	&\mathbf w(t+\frac{1}{\nu},x) = \mathbf w(t,x-1),
    \end{cases}
	\end{equation}
	for $t\geq 0$, $x\in\R$.  Define
	\begin{equation}
	\mathbf w(-t,x) = \mathbf w\left(-t+\frac{\lceil \nu t \rceil}{\nu}, x+\lceil \nu t \rceil\right),\quad \mbox{for } t>0,\, x\in\R.
	\end{equation}
	By the periodicity of $\mathbf w$, we get that $\mathbf w(t,x)$ satisfies \eqref{KPPequation_w}.
	Put
	\begin{equation}\label{v_w_relation}
	\widetilde {\mathbf w}(t,x):= \frac{e^{\lambda x-\gamma(\lambda)t}\mathbf w(t,x)} {\psi(x,\lambda)}.
	\end{equation}
	Recall that $\{Y_t, \Pi_x^{\lambda} \}$ is a diffusion with infinitesimal generator \eqref{Y_infin}.
	Define
	\begin{equation}\label{def-Y}
	f(Y_t) = \widetilde {\mathbf w}(-t,Y_t) e^{-\int_0^t \mathbf g(Y_s)\mathbf w(-s,Y_s)\mathrm{d}s},\quad t\geq 0.
	\end{equation}
	In this step we prove that
	$\{(f(Y_t))_{t\geq 0}, \Pi_x^{\lambda} \}$
	is a positive martingale.
	
	By the Feynman-Kac formula, we have
	\begin{equation}\label{w_feynman}
	\mathbf w(T,x) = \Pi_x \left[ \mathbf w(T-t,B_t) e^{\int_0^t \mathbf g(B_s)(1-\mathbf w(T-s,B_s))\mathrm{d}s} \right], \quad \mbox{for $T\in\R$, $t>0$.}
	\end{equation}
	Recall that,  since $m=1$,
	\begin{equation}
	\frac{\mathrm{d}\Pi_x^{\lambda}}{\mathrm{d}\Pi_x}\bigg{|}_{\F_t^B} =
	\frac{\Xi_t(\lambda)}{\Xi_0(\lambda)}
	= \frac{e^{-\gamma(\lambda)t-\lambda B_t +\int_0^t\mathbf g(B_s)\mathrm{d}s} \psi(B_t,\lambda) }{e^{-\lambda x} \psi(x,\lambda)}.
	\end{equation}
	Therefore
	\begin{align}
	\mathbf w(T,x) &=
	\Pi_x^{\lambda} \left[\frac{\Xi_0(\lambda)}{\Xi_t(\lambda)} \mathbf w(T-t,B_t) e^{\int_0^t \mathbf g(B_s)(1-\mathbf w(T-s,B_s))\mathrm{d}s} \right]\\
	&= \Pi_x^{\lambda} \left[ e^{-\lambda x} \psi(x,\lambda) \frac{e^{\lambda B_t + \gamma(\lambda)t - \int_0^t \mathbf g(B_s)\mathrm{d}s} \mathbf w(T-t,B_t)}{\psi(B_t,\lambda)} e^{\int_0^t \mathbf g(B_s)(1-\mathbf w(T-s,B_s))\mathrm{d}s} \right]\\
	&= \Pi_x^{\lambda} \left[ e^{-\lambda x +\gamma(\lambda)T } \psi(x,\lambda)   \frac{e^{\lambda B_t-\gamma(\lambda)(T-t)} \mathbf w(T-t,B_t)}{\psi(B_t,\lambda)} e^{-\int_0^t \mathbf g(B_s)\mathbf w(T-s,B_s)\mathrm{d}s} \right].
	\end{align}
	Thus we have
	\begin{equation}
	\widetilde {\mathbf w}(T,x) = \Pi_x^{\lambda} \left[\widetilde {\mathbf w}(T-t,B_t) e^{-\int_0^t \mathbf g(B_s)\mathbf w(T-s,B_s)\mathrm{d}s} \right].
	\end{equation}
	Note that both $\{B_t, \Pi_x^{\lambda}\}$ and $\{Y_t, \Pi_x^{\lambda}\}$ are diffusions with infinitesimal generator $\mathcal{A}$.
	Thus
	\begin{equation}\label{v_represent}
	\widetilde {\mathbf w}(T,x) = \Pi_x^{\lambda} \left[\widetilde {\mathbf w}(T-t,Y_t) e^{-\int_0^t \mathbf g(Y_s)\mathbf w(T-s,Y_s)\mathrm{d}s} \right].
	\end{equation}
	It follows from $\nu=\frac{\gamma(\lambda)}{\lambda}$ that
	\begin{equation}\label{v_periodic}
	\widetilde {\mathbf w}(t+\frac{1}{\nu},x+1) = \frac{e^{\lambda (x+1)-\gamma(\lambda)(t+\frac{1}{\nu})}\mathbf w(t+\frac{1}{\nu},x+1)} {\psi(x+1,\lambda)} = \frac{e^{\lambda x-\gamma(\lambda)t}\mathbf w(t,x)} {\psi(x,\lambda)} = \widetilde {\mathbf w}(t,x).
	\end{equation}
	For $0<s<t$, we have
	\begin{align}
	\Pi_{x}^{\lambda} \left[f(Y_t)|\F_s \right]&=
	\Pi_{x}^{\lambda} \left[\widetilde {\mathbf w}(-t,Y_t) e^{-\int_0^t \mathbf g(Y_r)\mathbf w(-r,Y_r)\mathrm{d}r}|\F_s \right]  \\
	&=e^{-\int_0^s \mathbf g(Y_r)\mathbf w(-r,Y_r)\mathrm{d}r}
	\Pi_{Y_s}^{\lambda} \left[\widetilde {\mathbf w}(-t,Y_{t-s}) e^{-\int_0^{t-s} \mathbf g(Y_r)\mathbf w(-(r+s),Y_r)\mathrm{d}r} \right]\\
	&=e^{-\int_0^s \mathbf g(Y_r)\mathbf w(-r,Y_r)\mathrm{d}r}
	\Pi_{Y_s}^{\lambda} \left[\widetilde {\mathbf w}(-s-(t-s),Y_{t-s}) e^{-\int_0^{t-s} \mathbf g(Y_r)\mathbf w(-s-r,Y_r)\mathrm{d}r}
	\right]\\
	&=e^{-\int_0^s \mathbf g(Y_r)\mathbf w(-r,Y_r)\mathrm{d}r}  \widetilde {\mathbf w}(-s,Y_s)=f(Y_s),
	\end{align}
	where the penultimate equality
	follows from \eqref{v_represent} with $T=-s$.
	Hence $\{(f(Y_t))_{t\geq 0}, \Pi_x^{\lambda} \}$ is a positive martingale.
	
	\textbf{Step 2} Suppose $L=1$.
It follows from  \eqref{lemma_slln} and \eqref{lemma_com(2)},
that
	$\lim\limits_{s\rightarrow\infty}\dfrac{Y_s+\nu s}{s} = -\gamma'(\lambda) + \frac{\gamma(\lambda)}{\lambda}>0$. Thus $\lim\limits_{s\rightarrow\infty} (Y_s+\nu s) = \infty$.
	Since a positive martingale has a non-negative finite limit, taking logarithms of \eqref{def-Y} and dividing by $Y_t+\nu t$ gives
	\begin{equation}\label{log-v}
	\limsup_{t\rightarrow\infty} \left\{\frac{\ln \widetilde {\mathbf w}(-t,Y_t)}{Y_t+\nu t} - \frac{1}{Y_t+\nu t} \int_0^t \mathbf g(Y_s)\mathbf w(-s,Y_s)\mathrm{d}s \right\}  \leq 0 \quad \Pi_x^{\lambda}\mbox{-a.s.}
	\end{equation}
	Put $\|\mathbf g\|_\infty=\max_{x\in[0,1]}\mathbf  g(x)$.
	Taking $T=t$ in  \eqref{w_feynman}, we get
	\begin{align}
	\mathbf w(t,x) &= \Pi_x \left[ \mathbf w(0,B_t) e^{\int_0^t \mathbf g(B_s)(1-\mathbf w(t-s,B_s))\mathrm{d}s} \right] \leq \Pi_x \left[ \mathbf w(0,B_t) e^{\|\mathbf g\|_\infty t} \right]\\
	&\leq e^{\|\mathbf g\|_\infty/\nu} \Pi_0 \left[ \mathbf w(0,B_t+x) \right],
	\quad t\in\left[0,\frac{1}{\nu}\right].
	\end{align}
	Since $\mathbf w(0,x)\rightarrow 0$ as $x\rightarrow\infty$, we have, for any $\epsilon>0$, $\mathbf w(0,x/2)\leq \epsilon/2$ when $x$ is large enough.
	Since $\Pi_0(B_t+x\leq x/2) = \Pi_0(B_t\leq -x/2)$, we have $\Pi_0(B_t+x\leq x/2)\leq \epsilon/2$ for $x$ large enough.  Therefore, for $x$ large enough,
	\begin{equation}
	\mathbf w(t,x) \leq e^{\|\mathbf g\|_\infty/\nu} \Pi_0(\mathbf w(0,B_t+x)) \leq e^{\|\mathbf g\|_\infty/\nu} \epsilon,\quad t\in\left[0,\frac{1}{\nu}\right].
	\end{equation}
	This implies that $\mathbf w(t,x)\rightarrow 0$ as $x\rightarrow\infty$ uniformly in $t\in\left[0,\frac{1}{\nu}\right]$. So
	\begin{equation}
	\frac{1}{t} \int_0^t \mathbf g(Y_s)\mathbf w(-s,Y_s)\mathrm{d}s \rightarrow 0, \quad \Pi_x^{\lambda}\mbox{-a.s.}
	\end{equation}
	Therefore, by \eqref{log-v},
	\begin{equation}
	\limsup_{t\rightarrow\infty} \left\{\frac{\ln \widetilde {\mathbf w}(-t,Y_t)}{Y_t+\nu t}\right\} \leq 0,  \quad \Pi_x^{\lambda}\mbox{-a.s.}
	\end{equation}
	Hence by \eqref{v_w_relation}, we have
	\begin{equation}\label{limsup-lnw}
	\limsup_{t\rightarrow\infty} \left\{\frac{\ln \mathbf w(-t,Y_t)}{Y_t+\nu t}\right\} \leq -\lambda, \quad \Pi_x^{\lambda}\mbox{-a.s.}
	\end{equation}
	This implies that, for any $\delta>0$ and $\Pi_x^{\lambda}$-a.s all $\omega$, there exists $C(\omega)>0$ such that
	\begin{equation}\label{w-lambda-delta}
	\mathbf w(-t,Y_t(\omega)) \leq C(\omega)e^{-(\lambda-\delta)(Y_t(\omega)+\nu t)}.
	\end{equation}
	Therefore,
	$$\int_0^\infty \mathbf g(Y_s)\mathbf w(-s,Y_s)\mathrm{d}s<+\infty\quad
	\Pi_x^{\lambda}\mbox{-a.s.}
	$$
	Consequently,
	by \eqref{def-Y},
	$\widetilde {\mathbf w}(-t,Y_t)$ converges $\Pi_x^{\lambda}$-almost surely to some limit, say $\xi_x$.

	Next we  use a coupling method to prove that $\xi_x$ is a constant $\Pi_x^{\lambda}$-almost surely.
	Consider $\{(Y_t^1,Y_t^2), t\geq 0; \widetilde{\Pi}_{(x,y)}^{\lambda}\}$ with
	$\{Y_t^1, t\geq 0\}$ and $\{Y_t^2, t\geq 0\}$ being independent, and
	\begin{equation}
	\{Y_t^1, t\geq 0; \widetilde{\Pi}_{(x,y)}^{\lambda}\} \overset{d}{=} \{Y_t, t\geq 0;\Pi_{x}^{\lambda}\}, \quad
	\{Y_t^2,t\geq 0; \widetilde{\Pi}_{(x,y)}^{\lambda}\} \overset{d}{=} \{Y_t,t\geq 0;\Pi_{y}^{\lambda}\}.
	\end{equation}
	Define
	\begin{equation}
	M_t^i = h(Y_t^i) + \gamma'(\lambda)t - h(Y_0^i), \quad i=1,2.
	\end{equation}
	Then $\{M_t^1, t\geq 0; \widetilde{\Pi}_{(x,y)}^{\lambda}\}$ and $\{M_t^2, t\geq 0; \widetilde{\Pi}_{(x,y)}^{\lambda}\}$ are independent martingales. Hence
	\begin{align}
	\langle M^1-M^2 \rangle_t = \langle M^1 \rangle_t + \langle M^2 \rangle_t \rightarrow \infty \quad \mbox{as } t\rightarrow\infty.
	\end{align}
	By the Dambis-Dubins-Schwarz theorem, we get
	\begin{align}\label{Mt_delta}
	\liminf_{t\rightarrow\infty} (M_t^1 - M_t^2) = -\infty \mbox{ and }
	\limsup_{t\rightarrow\infty} (M_t^1 - M_t^2) = +\infty,
	\quad \widetilde{\Pi}_{(x,y)}^{\lambda}\mbox{-a.s.}
	\end{align}
	Since $h(x) = x - \frac{\psi_{\lambda}(x,\lambda)}{\psi(x,\lambda)}$, we have
	\begin{align}
	Y_t^1 - Y_t^2 = M_t^1 - M_t^2 + \frac{\psi_{\lambda}(Y_t^1,\lambda)}{\psi(Y_t^1,\lambda)} - \frac{\psi_{\lambda}(Y_t^2,\lambda)}{\psi(Y_t^2,\lambda)}+ h(Y_0^1) - h(Y_0^2).
	\end{align}
	Combining \eqref{Mt_delta} with the boundness of  $\frac{\psi_{\lambda}(x,\lambda)}{\psi(x,\lambda)}$, we get
	\begin{align}
	\liminf_{t\rightarrow\infty} (Y_t^1 - Y_t^2) = -\infty \mbox{ and }
	\limsup_{t\rightarrow\infty} (Y_t^1 - Y_t^2) = +\infty,
	\quad \widetilde{\Pi}_{(x,y)}^{\lambda}\mbox{-a.s.}
	\end{align}
	Define $E:=\{\omega: \exists t_n=t_n(\omega)\to\infty \mbox{ with } Y_{t_n}^1 = Y_{t_n}^2 \mbox{ for all } n\}$. Then it follows from the display above that
	\begin{equation}\label{Yt_delta}
	\widetilde{\Pi}_{(x,y)}^{\lambda} (E)=1.
	\end{equation}
	If we use $\tilde{\xi}_x$ and $\tilde{\xi}_y$ to denote the limits of $\widetilde {\mathbf w}(-t,Y_t^1)$ and $\widetilde {\mathbf w}(-t,Y_t^2)$ under $\widetilde{\Pi}_{(x,y)}^{\lambda}$ respectively,
	then \eqref{Yt_delta} implies $\widetilde{\Pi}_{(x,y)}^{\lambda}(\tilde{\xi}_x = \tilde{\xi}_y) = 1$.
	Since  $\tilde{\xi}_x$ and $\tilde{\xi}_y$ are independent, there is a constant $\beta\geq 0$ such that  $\tilde{\xi}_x = \tilde{\xi}_y=\beta$.
	Since  $\tilde{\xi}_x \overset{d}{=} \xi_x$, we have for any $x\in\R$, $\xi_x = \beta$, which says
	\begin{equation}
	\widetilde {\mathbf w}(-t,Y_t) \rightarrow \beta,\quad \Pi_x^{\lambda}\mbox{-a.s.}
	\end{equation}
	
	Now we consider
	\begin{equation}
	f(Y_t,T) = \widetilde {\mathbf w}(T-t,Y_t) e^{-\int_0^t \mathbf g(Y_s)\mathbf w(T-s,Y_s)\mathrm{d}s},\quad t\geq 0.
	\end{equation}
	The proof in Step 1 also works if $f(Y_t)$ is replaced by $f(Y_t,T)$.
	Then using the same argument as above, there exists another constant $\beta_T$ such that
	\begin{equation}
	\widetilde {\mathbf w}(T-t,Y_t) \rightarrow \beta_T,\quad \Pi_x^{\lambda}\mbox{-a.s.}
	\end{equation}
	We also have
	\begin{align}
	\liminf_{t\rightarrow\infty} \left(Y_{t+T}^1 - Y_t^2\right) = -\infty \mbox{ and }
	\limsup_{t\rightarrow\infty} \left(Y_{t+T}^1 - Y_t^2\right) = +\infty.
	\end{align}
	Hence, if we put $E_T:=\left\{\omega: \exists t_n=t_n(\omega)\to\infty \mbox{ with } Y_{t_n+T}^1 = Y_{t_n}^2 \mbox{ for all } n\right\}$, then
	\begin{equation}\label{ET}
	\widetilde{\Pi}_{(x,y)}^{\lambda} (E_T)=1.
	\end{equation}
	Notice that $Y_{t_n+T}^1 = Y_{t_n}^2$ implies
	$\widetilde {\mathbf w}\left(T-(t+T),Y_{t+T}^1\right) = \widetilde {\mathbf w}(-t,Y_t^2)$.
	Combining this with \eqref{ET}, we have
	$\beta_T = \beta$,
	that is, for any $x,T\in\R$, it holds that
	\begin{equation}\label{v_Yt_converge}
	\widetilde {\mathbf w}(T-t,Y_t) \rightarrow \beta,\quad \Pi_x^{\lambda}\mbox{-a.s.}
	\end{equation}
	
	\textbf{Step 3}
	First, we prove that $\widetilde {\mathbf w}$ is bounded.
	Using the same notation in Step 2, we also have for any $k\in \mathbb{Z}$
	\begin{equation}
	\left\langle M^1_{\cdot+\frac{k}{\nu}}-M^2_{\cdot} \right\rangle_t = \left\langle M^1_{\cdot+\frac{k}{\nu}} \right\rangle_t + \left\langle M^2_{\cdot} \right\rangle_t \rightarrow \infty \quad \mbox{as } t\rightarrow\infty.
	\end{equation}
	By the Dambis-Dubins-Schwarz theorem, it holds that
	\begin{equation}
	\liminf_{t\rightarrow\infty} \left(M_{t+\frac{k}{\nu}}^1 - M_t^2\right) = -\infty \mbox{ and }
	\limsup_{t\rightarrow\infty} \left(M_{t+\frac{k}{\nu}}^1 - M_t^2\right) = +\infty,
	\end{equation}
	so the same conclusion holds for $Y_{t+\frac{k}{\nu}}^1 - Y_t^2$.
	This implies, with $E_k:=\{\omega: \exists t_n=t_n(\omega)\to\infty \mbox{ with } Y_{t_n+\frac{k}{\nu}}^1+k = Y_{t_n}^2 \mbox{ for all } n\}$,
	\begin{equation}\label{Yt_delta_k}
	\widetilde{\Pi}_{(x,y)}^{\lambda} (E_k)=1.
	\end{equation}
	Consider
	$$
	\Omega_0 = \bigcap_{k\in\mathbb{Z}} E_k
	\bigcap\left\{\lim_{t\rightarrow\infty}\widetilde {\mathbf w}(-t,Y_t^i) = \beta,\, \lim_{t\rightarrow\infty}\frac{Y_t^i}{t} = -\gamma'(\lambda), \mbox{ and } Y_t^i
	{\mbox{ is continuous for } i=1,2} \right\}.
	$$
	By \eqref{lemma_slln}, \eqref{v_Yt_converge} and \eqref{Yt_delta_k},
    we get $\widetilde{\Pi}_{(x,y)}^{\lambda}(\Omega_0) = 1$. By \eqref{v_periodic}, it suffices to show that $\widetilde {\mathbf w}(t,x)$ is bounded in $[0,\frac{1}{\nu}] \times \R$. Fix $\omega\in\Omega_0$, consider the continuous curves
	\begin{equation}
	\mathcal{L}_k^i = \left\{\left(-t+\frac{k}{\nu}, Y_t^i(\omega)+k\right): t\in\left(\frac{k-1}{\nu},\frac{k}{\nu}\right) \right\}, \quad k\in\N,\, i=1,2.
	\end{equation}
	In Step 2 we have shown $\lim_{t\rightarrow\infty} (Y_t^i(\omega)+\nu t) = \infty$ for $i=1,2$,
	thus for any $x_0$ large enough,  there exist $k=k(\omega)>j=j(\omega)\in N$ such that
	\begin{equation}
	x_0 \leq Y_t^1(\omega)+k,\mbox{ for any } t\in \left[\frac{k-1}{\nu},\frac{k}{\nu}\right], \mbox{ and }
	x_0 \geq Y_t^2(\omega)+j,\mbox{ for any } t\in \left[\frac{j-1}{\nu},\frac{j}{\nu}\right].
	\end{equation}
	Define
	\begin{equation}
	\widetilde{\tau}(\omega) = \inf\left\{t\geq j/\nu:  Y_{t+\frac{k-j}{\nu}}^1(\omega)+k = Y_t^2(\omega)+j  \right\}.
	\end{equation}
	By the definition of $\Omega_0$, we know $\widetilde{\tau}(\omega)<+\infty$.
	Let $\widetilde{\mathcal{L}}$ denote the line segment $\{\frac{1}{\nu}\} \times [Y_{\frac{j-1}{\nu}}^2(\omega)+j, Y_{\frac{k-1}{\nu}}^1(\omega)+k]$, and define the curves
	\begin{align}
	&\widetilde{\mathcal{L}}_k^1 = \left\{\left(-t+\frac{k}{\nu}, Y_t^1(\omega)+k\right): t\in \left[\frac{k-1}{\nu}, \widetilde{\tau}(\omega)+\frac{k-j}{\nu}\right] \right\},\\
	&\widetilde{\mathcal{L}}_j^2 = \left\{\left(-t+\frac{j}{\nu}, Y_t^2(\omega)+j\right): t\in \left[\frac{j-1}{\nu}, \widetilde{\tau}(\omega)\right] \right\}.
	\end{align}
	By the definition of $\widetilde{\tau}(\omega)$, we have
	\begin{equation}
	\left(-\left(\widetilde{\tau}(\omega)+\frac{k-j}{\nu}\right) + \frac{k}{\nu}, \, Y_{\widetilde{\tau}(\omega)+\frac{k-j}{\nu}}^1(\omega)+k\right) = \left(-\widetilde{\tau}(\omega)+\frac{j}{\nu}, \, Y_{\widetilde{\tau}(\omega)}^2(\omega)+j\right).
	\end{equation}
	Combining \eqref{Y_infin}, \eqref{v_represent} and the Feynman-Kac formula, we have
	\begin{equation}\label{v_equation}
	\frac{\partial \widetilde {\mathbf w}}{\partial t} = \frac{1}{2} \frac{\partial^2 \widetilde {\mathbf w}}{\partial x^2} + (\frac{\psi_x(x,\lambda)}{\psi(x,\lambda)} - \lambda) \frac{\partial \widetilde {\mathbf w}}{\partial x} -
    \mathbf g \mathbf w \widetilde {\mathbf w}.
	\end{equation}
		\begin{figure*}
		\centering
				\includegraphics[width=12cm]{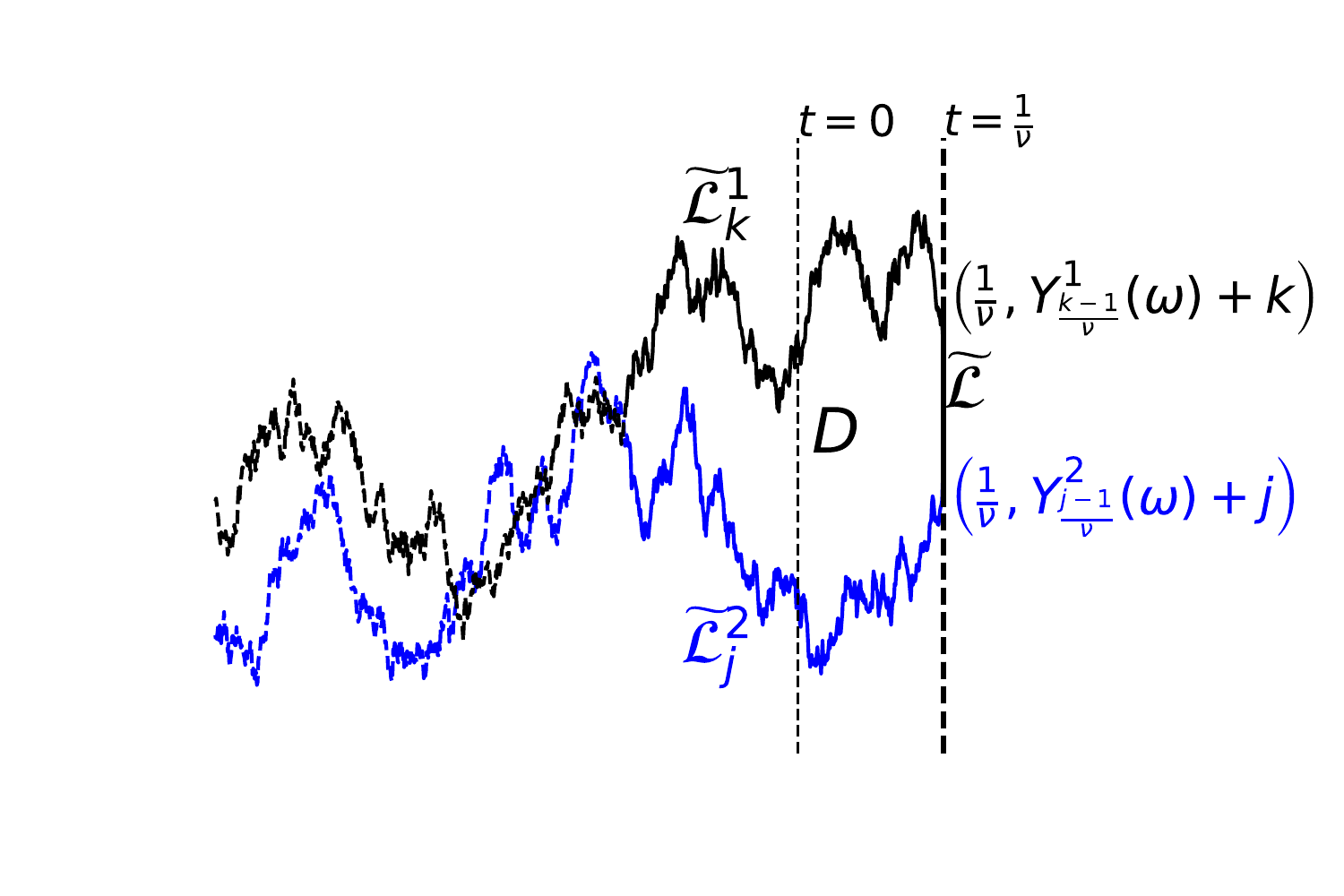}
		\caption{
Bounded domain
$D$ with boundary $\widetilde{\mathcal{L}}\cup\widetilde{\mathcal{L}}_k^1\cup\widetilde{\mathcal{L}}_j^2$}
	\end{figure*}
    	Let $D$ denote the bounded domain with boundary $\widetilde{\mathcal{L}}\cup\widetilde{\mathcal{L}}_k^1\cup\widetilde{\mathcal{L}}_j^2$
	(see Figure 1).
	By the maximum principle, we have $\widetilde {\mathbf w}$ attains its maximum in $\overline D$ on $\widetilde{\mathcal{L}}_k^1\cup\widetilde{\mathcal{L}}_j^2$, where $\overline D$ is the closure of $D$.
	Hence the maximum of $\widetilde {\mathbf w}$ on $\overline D$
	is less than or equal to $K:=\max_{t\geq 0} \{\widetilde {\mathbf w}(-t,Y_t^1(\omega)), \widetilde {\mathbf w}(-t,Y_t^2(\omega))\}$. By the continuity of $\widetilde {\mathbf w}$ and
	\begin{equation}
	\lim_{t\rightarrow\infty} \widetilde {\mathbf w}(-t,Y_t^i(\omega)) = \beta,  \quad i=1,2,
	\end{equation}
	we get $K<\infty$. Notice that for any
	$t_0\in [0,\frac{1}{\nu}]$ and $x_0$ large enough, $(t_0,x_0) \in \overline D$
	and so $\widetilde {\mathbf w}(t_0,x_0) \leq K$.
	Combining this with $\lim\limits_{x\rightarrow-\infty} \widetilde {\mathbf w}(t,x) = 0$ and the continuity of $\widetilde {\mathbf w}$, we get $\widetilde {\mathbf w}(t,x)$ is bounded in $[0,\frac{1}{\nu}]\times \R$, hence bounded in $\R\times \R$.
	
	Since $\widetilde {\mathbf w}(t,x)$ is bounded, by \eqref{v_represent}, \eqref{v_Yt_converge} and the dominated convergence theorem,
	\begin{equation}
	\widetilde {\mathbf w}(t,x) = \Pi_x^{\lambda} \left[\beta e^{-\int_0^{\infty} \mathbf g(Y_s)\mathbf w(t-s,Y_s)\mathrm{d}s} \right]\leq \beta.
	\end{equation}
	Since  $\widetilde {\mathbf w}(t,x)>0$, we have $\beta>0$.
	
	Next we show that
	\begin{equation}\label{unif-bound-tilde-w}
	\widetilde {\mathbf w}(t,x) \rightarrow \beta \quad
	\mbox{ as } x\rightarrow\infty\mbox{ uniformly in }t\in \left[0,\frac{1}{\nu}\right].
	\end{equation}
	It follows from  \eqref{v_w_relation} that
	\begin{equation}\label{w_estimate}
	\begin{split}
	\mathbf w(t-s, Y_s) &\leq e^{-\lambda Y_s + \gamma(\lambda)(t-s)} \psi(Y_t,\lambda) \widetilde {\mathbf w}(t-s,Y_s)\\
	&\leq  \beta \max_{z\in[0,1]}\psi(z,\lambda) e^{\gamma(\lambda)\frac{1}{\nu}} e^{-\lambda(Y_s+\nu s)}  \\
	&\leq C_1 e^{-\lambda(Y_s+\nu s)},
	\quad \forall t\in \left[0,\frac{1}{\nu}\right],
	\end{split}
	\end{equation}
	where $C_1$ is a constant depends only  on $\lambda$. Moreover, we have
	\begin{equation}
	\mathbf g(Y_s) \mathbf w(t-s, Y_s) \leq C e^{-\lambda(Y_s+\nu s)}, \quad \forall t\in \left[0,\frac{1}{\nu}\right],
	\end{equation}
	where $C = C_1 \|\mathbf g\|_{\infty}$.
	For any $y\in [0,1]$ and $n\in\N$, let
	\begin{equation}
	f_n(y) = \Pi_y^{\lambda} \left[ e^{-\int_0^{\infty} Ce^{-\lambda(Y_s+n+\nu s)}\mathrm{d}s} \right].
	\end{equation}
	Then by the periodicity of $Y_t$, we have
	\begin{equation}
	\widetilde {\mathbf w}(t,x) \geq \beta \Pi_x^{\lambda} \left[ e^{-\int_0^{\infty} Ce^{-\lambda(Y_s+\nu s)} \mathrm{d}s} \right] = \beta \Pi_{\{x\}}^{\lambda} \left[ e^{-\int_0^{\infty} Ce^{-\lambda(Y_s+\lfloor x \rfloor +\nu s)} \mathrm{d}s} \right] =
	\beta f_{\lfloor x \rfloor }(\{x\}).
	\end{equation}
	Since  $\int_0^{\infty} e^{-\lambda(Y_s+\nu s)} \mathrm{d}s < +\infty$ $\Pi_x^{\lambda}$-a.s., by the dominated convergence theorem, we get
	\begin{align}
	\lim_{n\rightarrow\infty} f_n(y) &= \Pi_y^{\lambda} \left[\lim_{n\rightarrow\infty} e^{-\int_0^{\infty} Ce^{-\lambda(Y_s+n+\nu s)}\mathrm{d}s} \right]
	= \Pi_y^{\lambda} \left[ e^{- \lim_{n\rightarrow\infty} \int_0^{\infty} Ce^{- \lambda(Y_s+n+\nu s)}\mathrm{d}s} \right]\\
	&= \Pi_y^{\lambda} \left[ e^{- \int_0^{\infty} \lim_{n\rightarrow\infty} Ce^{- \lambda(Y_s+n+\nu s)}\mathrm{d}s } \right] = 1.
	\end{align}
	Notice that $f_n(y) \leq f_{n+1}(y)$, using Dini's theorem we get $f_n(y) \rightarrow 1$ uniformly for $y\in[0,1]$. Combining this with $\beta f_n(y) \leq \widetilde {\mathbf w}(t,y+n) \leq \beta$, we have \eqref{unif-bound-tilde-w}.
	
	By \eqref{v_w_relation}, we get
	\begin{equation}\label{u_uniform_conver}
	\frac{e^{\lambda x-\gamma(\lambda)t}\mathbf w(t,x)} {\psi(x,\lambda)} \rightarrow \beta \quad
	\mbox{ as } x\rightarrow\infty \mbox{ uniformly in } t\in \left[0,\frac{1}{\nu}\right].
	\end{equation}

	\textbf{Step 4} Suppose $L=1$.
	We  will show that
	\begin{equation}\label{e:step3}
	w\left(\frac{y-x}{\nu}, y\right) \sim \beta e^{-\lambda x}\psi(y,\lambda) \text{ as } x\rightarrow +\infty \text{ uniformly in } y\in [0,1].
	\end{equation}
	It is equivalent to show that, for any $\epsilon>0$, there exists $x_0$ such that for any $x>x_0$,
	\begin{equation}
	\sup_{y\in [0,1]} \bigg{|}\frac{e^{\lambda x}\mathbf w\left(\frac{y-x}{\nu}, y\right)}{\psi(y,\lambda)} - \beta\bigg{|} < \epsilon.
	\end{equation}
	By \eqref{u_uniform_conver}, for any $\epsilon>0$, there exists $z_0$ such that for any $z>z_0$,
	\begin{equation}
	\sup_{t\in [0,\frac{1}{v}]} \bigg{|}\frac{e^{\lambda z-\gamma(\lambda)t}\mathbf w\left(t, z\right)}{\psi(z,\lambda)} - \beta\bigg{|} < \epsilon.
	\end{equation}
	Put
	\begin{equation}
	t = \frac{\{y-x \}}{\nu},  \quad z=-\lfloor y-x \rfloor  + y,
	\end{equation}
	where $\{x\}:=x-\lfloor x \rfloor$. For any  $x>z_0+1$, we have
	\begin{equation}
	\sup_{\frac{\{y-x \}}{\nu}\in [0,\frac{1}{\nu}]} \bigg{|}\frac{e^{\lambda (-\lfloor y-x \rfloor +y)-\gamma(\lambda)\frac{\{y-x \}}{\nu}}\mathbf w\left(\frac{\{y-x \}}{\nu}, -\lfloor y-x \rfloor+y\right)} {\psi(-\lfloor y-x \rfloor+y,\lambda)} - \beta\bigg{|} < \epsilon,
	\end{equation}
	that is
	\begin{equation}
	\sup_{y\in [0,1]} \bigg{|}\frac{e^{\lambda x}\mathbf w\left(\frac{y-x}{\nu}, y\right)}{\psi(y,\lambda)} - \beta\bigg{|} < \epsilon,
	\end{equation}
	where we used the periodicity of $\psi$ and the fact that
	\begin{equation}
	\lambda (-\lfloor y-x \rfloor +y) - \gamma(\lambda)\frac{\{y-x \}}{\nu} = \lambda (-\lfloor y-x \rfloor +y) - \lambda \{y-x \} = \lambda (-(y-x)+y) = \lambda x.
	\end{equation}
	Thus \eqref{e:step3} holds.

	\textbf{Step 5} For the general branching mechanism, $\mathbf w=1-\mathbf u$ satisfies
	\begin{equation}
	\frac{\partial \mathbf w}{\partial t} = \frac{1}{2} \frac{\partial^2 \mathbf w}{\partial x^2} +
    \mathbf g\cdot(1-\mathbf w-\mathbf f(1-\mathbf w)),
	\end{equation}
	where $\mathbf f(s) = \mathbf E s^{1+L} = \sum\limits_{k=0}^{\infty} p_k s^{1+k}$ and $m = \sum\limits_{k=0}^{\infty} k p_k$.
	By the Feynman-Kac formula,
	\begin{equation}
	\mathbf w(T,x) = \Pi_x \left[ \mathbf w(T-t,B_t) e^{\int_0^t \mathbf g(B_s) \frac{1-\mathbf w-\mathbf f(1-\mathbf w)}{\mathbf w}(T-s,B_s) \mathrm{d}s} \right].
	\end{equation}
	By the definition in \eqref{meas_Plambda},
	\begin{align}
	\mathbf w(T,x) &= \Pi_x^{\lambda}
	\left[\frac{\Xi_0(\lambda)}{\Xi_t(\lambda)} \mathbf w(T-t,B_t) e^{\int_0^t \mathbf g(B_s)\frac{1-\mathbf w-\mathbf f(1-\mathbf w)}{\mathbf w}(T-s,B_s)\mathrm{d}s} \right]\\
	&= \Pi_x^{\lambda} \left[ e^{-\lambda x +\gamma(\lambda)T } \psi(x,\lambda)   \frac{e^{\lambda B_t-\gamma(\lambda)(T-t)} \mathbf w(T-t,B_t)}{\psi(B_t,\lambda)} e^{\int_0^t \mathbf g(B_s)(\frac{1-\mathbf w-\mathbf f(1-\mathbf w)}{\mathbf w}(T-s,B_s)-m)\mathrm{d}s} \right].
	\end{align}
	Put
	\begin{equation}
	A(w) =\left\{\begin{array}{ll}  m - \frac{1-w-\mathbf f(1- w)}{w}, &\quad w\in(0,1],\\
	0, &\quad w=0.\end{array}\right.
	\end{equation}
	Since $(B_t, \Pi^\lambda_x)$ and $(Y_t, \Pi^\lambda_x)$ have the same law, by \eqref{v_w_relation}, we have
	\begin{align}
	\widetilde {\mathbf w}(T,x) =  \Pi_x^{\lambda} \left[\widetilde {\mathbf w}(T-t,Y_t) e^{-\int_0^t \mathbf g(Y_s)A(\mathbf w)(T-s,Y_s)\mathrm{d}s} \right].
	\end{align}
	It follows from \cite[Corollary 2, p.26]{Athreya} that $A(\cdot)$ is non-negative and non-decreasing.
	Moreover, for any $r,c\in (0,1)$,
	\begin{equation}\label{iff-llogl}
	\sum_{n=0}^{\infty} A(cr^n) <+\infty \quad \mbox{ iff } \quad \mathbf E(L\log^+L) < +\infty.
	\end{equation}
	Using the argument of Step 1, we get that
	$\{\widetilde {\mathbf w}(-t,Y_t) e^{-\int_0^t \mathbf g(Y_s)A(\mathbf w)(-s,Y_s)\mathrm{d}s}, \Pi_x^{\lambda}\}_{t\geq 0}$ is a non-negative martingale.
	Using the argument at the beginning of Step 2,  we get
	$\mathbf w(t,x) \rightarrow 0$ as $x\rightarrow\infty$ uniformly in $t\in [0,\frac{1}{\nu}]$,
	$\mathbf w(-t,Y_t)$ decays exponentially with rate at least $-\lambda$.
	So $\widetilde {\mathbf w}(-t,Y_t)$ converges $\Pi_x^{\lambda}$-almost surely, if we can show that
	\begin{align}
	\int_0^{\infty}  \mathbf g(Y_s)A(\mathbf w)(-s,Y_s)\mathrm{d}s < +\infty, \quad \mbox{ $\Pi_x^{\lambda}$-a.s.}
	\end{align}
	By \eqref{w-lambda-delta} and \eqref{iff-llogl},
	if $\mathbf E(L\log^+L) < \infty$, we have
	\begin{align}
	\int_0^{\infty}  \mathbf g(Y_s)A(\mathbf w)(-s,Y_s)\mathrm{d}s &\leq \|\mathbf g\|_\infty \int_0^{\infty} A(C e^{-(\lambda-\delta)(Y_s+\nu s) }) \mathrm{d}s \\
	&\leq \|\mathbf g\|_\infty \int_0^{\infty} A(C (e^{-(\lambda-\delta)(\nu-\gamma'(\lambda))s}))\mathrm{d}s\\
	&\leq \|\mathbf g\|_\infty \sum_{n=0}^{\infty} A(C (e^{-(\lambda-\delta)(\nu-\gamma'(\lambda))n})) < +\infty \quad \mbox{ $\Pi_x^{\lambda}$-a.s. }
	\end{align}
	where  $C$ is a constant depending on $\omega$ and may change values from line to line.

	Also by the arguments in Steps 2 and 3,
	$\widetilde {\mathbf w}(t,x)$ satisfies
	\begin{equation}
	\frac{\partial \widetilde {\mathbf w}}{\partial t} = \frac{1}{2} \frac{\partial^2 \widetilde {\mathbf w}}{\partial x^2} + (\frac{\psi_x(x,\lambda)}{\psi(x,\lambda)} - \lambda) \frac{\partial \widetilde {\mathbf w}}{\partial x} -
    \mathbf g \cdot A({\mathbf w}) \widetilde {\mathbf w},
	\end{equation}
	and the maximum principle holds.
	Hence, $\widetilde {\mathbf w}(t,x)$ is bounded in $[0,\frac{1}{\nu}]\times \R$.
	Since $A(\cdot)$ is non-decreasing,
	\eqref{w_estimate} implies
	\begin{equation}
	A(\mathbf w)(t-s,Y_s) \leq A(C_1e^{-\lambda(Y_s+\nu s)}).
	\end{equation}
	By the periodicity of $\{Y_t\}$, we have
	\begin{align}
	\widetilde  {\mathbf w}(t,y+n) \geq \beta \Pi_{y+n}^{\lambda} \left[ e^{-\int_0^{\infty} \|\mathbf g\|_\infty A(C_1e^{-\lambda(Y_s+\nu s)}) \mathrm{d}s } \right] = \beta \Pi_y^{\lambda} \left[ e^{-\int_0^{\infty} \|\mathbf g\|_\infty A(C_1e^{-\lambda(Y_s+n+\nu s)}) \mathrm{d}s } \right].
	\end{align}
	To prove the theorem, it suffices to show
	\begin{equation}\label{fny_converge}
	\Pi_y^{\lambda} \left[ e^{-\int_0^{\infty} \|\mathbf g\|_\infty A(C_1e^{-\lambda(Y_s+n+\nu s)}) \mathrm{d}s } \right] \rightarrow 1, \quad \mbox{ as $n\rightarrow\infty$ uniformly for $y\in [0,1]$.}
	\end{equation}
	We know
	\begin{equation}
	\int_0^{\infty} \|\mathbf g\|_\infty A(C_1e^{-\lambda(Y_s+n+\nu s)}) \mathrm{d}s < +\infty  \quad  \mbox{ $\Pi_x^{\lambda}$-a.s.}
	\end{equation}
	Using an argument similar to that in Step 3, \eqref{fny_converge} follows from
	$\lim_{u\downarrow 0} A(u) = 0$,
		the dominated convergence theorem and Dini's theorem.
	This completes the proof.
\end{proof}

\section{Proof of Theorem \ref{thrm_asym_crit}}
In this section, we prove the asymptotic behavior in the critical case.
\begin{proof}[Proof of Theorem \ref{thrm_asym_crit}]
We  prove the theorem in seven steps. In the first six steps, we consider the case that $L=1$. In the last step we consider general $L$.

\textbf{Step 1} Suppose $L=1$. Put $\mathbf w(t,x) = 1 -\mathbf u(t,x)$. We first prove that  $\mathbf w(t,x)$ decays exponentially with rate at least $-\lambda$ uniformly in $t\in [0,\frac{1}{\nu}]$.
Using an argument similar to that of Theorem \ref{thrm_asym_super}, we have that, for $t, x\in\R$, $\mathbf w(t,x)$ satisfies
\begin{equation}\label{KPPequation_w_crit}
\begin{cases}
&\frac{\partial \mathbf w}{\partial t} = \frac{1}{2} \frac{\partial^2 \mathbf w}{\partial x^2} +
\mathbf g\cdot(\mathbf w-\mathbf w^2),\\
&\mathbf w(t+\frac{1}{\nu^*},x) = \mathbf w(t,x-1).
\end{cases}
\end{equation}
By the Feynman-Kac formula, we get
\begin{equation}
\mathbf w(T,x) = \Pi_x \left[\mathbf w(T-t,B_t) e^{\int_0^t \mathbf g(B_s)(1-\mathbf w(T-s,B_s))\mathrm{d}s} \right], \quad \mbox{for $T\in\R$, $t>0$.}
\end{equation}
For any $\lambda<\lambda^*$, define
\begin{equation}\label{vtilde_w_relation}
\widetilde{\mathbf w}(t,x) = \frac{e^{\lambda x-\gamma(\lambda)t}\mathbf w(t,x)} {\psi(x,\lambda)}.
\end{equation}
Changing measure with $\Xi_t(\lambda)$
and following the same ideas in Step 1 of the proof of Theorem \ref{thrm_asym_super}, we get that
\begin{equation}\label{vtilde_represent}
\widetilde{\mathbf w}(T,x) = \Pi_x^{\lambda} \left[\widetilde{\mathbf w}(T-t,Y_t) e^{-\int_0^t \mathbf g(Y_s)\mathbf w(T-s,Y_s)\mathrm{d}s} \right],
\end{equation}
 and $\{\widetilde{\mathbf w}(-t,Y_t) e^{-\int_0^t \mathbf g(Y_s) \mathbf w(-s,Y_s)\mathrm{d}s}, \Pi_x^{\lambda} \}_{t\geq 0}$
is a positive martingale.
Therefore, we have
\begin{equation}
\limsup_{t\rightarrow\infty} \left\{\frac{\ln \widetilde{\mathbf w}(-t,Y_t)}{Y_t+\nu^*t} - \frac{1}{Y_t+\nu^*t} \int_0^t \mathbf g(Y_s)\mathbf w(-s,Y_s)\mathrm{d}s \right\}  \leq 0 \quad
\Pi_x^{\lambda}\mbox{-a.s.}
\end{equation}
Notice that for $\lambda<\lambda^*$,
\begin{align}
\lim_{t\rightarrow\infty} \frac{Y_t+\nu^*t}{t} \rightarrow -\gamma'(\lambda) + \nu^* > 0.
\end{align}
Combining this with \eqref{vtilde_w_relation}, we get
\begin{align}
\limsup_{t\rightarrow\infty} \frac{\ln \widetilde{\mathbf w}(-t,Y_t)}{Y_t+\nu^*t} &=  \limsup_{t\rightarrow\infty} \frac{\ln(e^{\lambda(Y_t+\nu^*t) + (\gamma(\lambda)-\lambda \nu^*)t} \mathbf w(-t,Y_t))}{Y_t+\nu^*t}\\ &= \limsup_{t\rightarrow\infty}  \frac{\ln \mathbf w(-t,Y_t)}{Y_t+\nu^*t} + \lambda + \frac{\gamma(\lambda)-\lambda \nu^*}{\nu^*-\gamma'(\lambda)}
\leq 0 \quad
\Pi_x^{\lambda}\mbox{-a.s.}
\end{align}
where $\gamma(\lambda) - \lambda \nu^*>0$.
Hence $\mathbf w(-t,Y_t)$ decays exponentially with rate at least $-\lambda$ $\Pi_x^{\lambda}$-almost surely. This implies
\begin{equation}
\int_0^\infty \mathbf g(Y_s)\mathbf w(-s,Y_s)\mathrm{d}s<+\infty\quad
\Pi_x^{\lambda}\mbox{-a.s.}
\end{equation}
Then, $\widetilde {\mathbf w}(-t,Y_t)$ converges $\Pi_x^{\lambda}$-almost surely. Using a coupling method similar to that of Step 2 in the proof of Theorem \ref{thrm_asym_super}, we get that the limit of $\widetilde {\mathbf w}(-t,Y_t)$ is a constant. Notice that
\begin{equation}
\widetilde {\mathbf w}(t+\frac{1}{\nu^*},x+1) = \frac{e^{\lambda (x+1)-\gamma(\lambda)(t+\frac{1}{\nu^*})}\mathbf w(t+\frac{1}{\nu^*},x+1)} {\psi(x+1,\lambda)} = e^{\lambda-\frac{\gamma(\lambda)}{\nu^*}} \widetilde  {\mathbf w}(t,x) < \widetilde {\mathbf w}(t,x),
\end{equation}
where we used $\mathbf w(t+\frac{1}{\nu^*},x+1) = \mathbf w(t,x)$ and $\nu^* < \frac{\gamma(\lambda)}{\lambda}$ for $\lambda<\lambda^*$. Therefore, for any $k\in\N$,
\begin{equation}
\widetilde {\mathbf w}(-t+\frac{k}{\nu},Y_t+k) \leq \widetilde {\mathbf w}(-t,Y_t).
\end{equation}
Using an argument similar to that of Step 3 in the proof of Theorem \ref{thrm_asym_super}, we get $\widetilde {\mathbf w}(t,x)$ is bounded in
$[0,\frac{1}{\nu^*}]\times \R$.
Then, by \eqref{vtilde_w_relation}, $\mathbf w(t,x)$ decays exponentially with rate at least  $-\lambda$ uniformly in
$t\in [0,\frac{1}{\nu^*}]$.

\textbf{Step 2}
Recall the definitions of $h$, $\tau_{\lambda}^x$, $\Lambda_t^{(x,\lambda)}$
and $\Pi^{(x,\lambda)}_y$ given in \eqref{def-h}, \eqref{def_tauB}, \eqref{mart_Lambda} and \eqref{meas_change} respectively,
with $\lambda = \lambda^*$ in \eqref{def-h}.
Fix $y\in \R$. For any $(t,x)$ such that $y-\gamma'(\lambda^*)t+h(x)>0$, define
\begin{equation}\label{v_w_relation_crit}
\widehat {\mathbf w}(t,x,y) := \frac{e^{\lambda^*x-\gamma(\lambda^*)t} \mathbf w(t,x)} {\psi(x,\lambda^*)(y-\gamma'(\lambda^*)t+h(x))},
\end{equation}
and for any $z>0$
\begin{equation}\label{tau-z}
\tau_z := \inf\{t\geq 0: y+\gamma'(\lambda^*)t+h(B_t) \leq z \}.
\end{equation}
We mention that $\tau_z$ actually depends on $y$.
For any $x\in\R$, we may define
\begin{equation}\label{tau-z-x}
\tau_z(x) := \inf\{t\geq 0: x+\gamma'(\lambda^*)t+h(B_t) \leq z \}.
\end{equation}
Then $\tau_z$ is a shorthand for $\tau_z(y)$, and
for any $x\in\R$, $\tau_z(y-x) = \tau_{z+x}$.
Using \eqref{v_w_relation_crit}, it is easy to show
\begin{equation}\label{v_periodic_crit}
\widehat{\mathbf w}(t+\frac{1}{\nu^*}, x+1, y) = \widehat{\mathbf w}(t,x,y).
\end{equation}
We first prove that  for any  $T\in\R$ and $t>0$,
\begin{equation}\label{v_feynman_crit}
\widehat {\mathbf w}(T,x,y+\nu^*T) = \Pi_x^{(y,\lambda^*)} \left( \widehat {\mathbf w}(T-t\wedge\tau_z, B_{t\wedge\tau_z}, y+\nu^*T) e^{-\int_0^{t\wedge\tau_z} \mathbf g(B_s) \mathbf w(T-s,B_s) \mathrm{d}s } \right).
\end{equation}
First note that, by the Feynman-Kac formula and the optional stopping theorem,
\begin{equation}\label{FK-stopping}
\mathbf  w(T,x)= \Pi_x \left(\mathbf w(T-t\wedge\tau_z, B_{t\wedge\tau_z}) e^{\int_0^{t\wedge\tau_z} \mathbf g(B_s) (1-\mathbf w(T-s,B_s)) \mathrm{d}s} \right),\quad T\in \R, t>0.
\end{equation}
Noticing that $\Lambda_{t\wedge\tau_z}^{(y,\lambda^*)} > 0$ and $\nu^* = \gamma'(\lambda^*)$, a direct calculation shows that for $x>h^{-1}(y)$,
\begin{align}
&\Pi_x^{(y,\lambda^*)} \left( \psi(x,\lambda^*)(y+h(x)) e^{-\lambda^*x+\gamma(\lambda^*)T} \widehat {\mathbf w}(T-t\wedge\tau_z, B_{t\wedge\tau_z}, y+\nu^*T) e^{-\int_0^{t\wedge\tau_z} \mathbf g(B_s) \mathbf w(T-s,B_s) \mathrm{d}s}  \right)\\
= &\Pi_x^{(y,\lambda^*)} \left( \frac{\Lambda_0^{(y,\lambda^*)}}{\Lambda_{t\wedge\tau_z}^{(y,\lambda^*)}} \mathbf w(T-t\wedge\tau_z, B_{t\wedge\tau_z}) e^{\int_0^{t\wedge\tau_z} \mathbf g(B_s) (1-\mathbf w(T-s,B_s)) \mathrm{d}s} \right)\\
= &\Pi_x \left( \frac{\Lambda_0^{(y,\lambda^*)}}{\Lambda_{t\wedge\tau_z}^{(y,\lambda^*)}} \mathbf w(T-t\wedge\tau_z, B_{t\wedge\tau_z}) e^{\int_0^{t\wedge\tau_z} \mathbf g(B_s) (1-\mathbf w(T-s,B_s)) \mathrm{d}s} \frac{\Lambda_{t\wedge\tau_z}^{(y,\lambda^*)}}{\Lambda_0^{(y,\lambda^*)}} \mathbf{1}_{\{\Lambda_{t\wedge\tau_z}^{(y,\lambda^*)}>0  \} } \right)\\
= &\Pi_x \left( \mathbf w(T-t\wedge\tau_z, B_{t\wedge\tau_z}) e^{\int_0^{t\wedge\tau_z} \mathbf g(B_s) (1-\mathbf w(T-s,B_s)) \mathrm{d}s} \right)
= \mathbf  w(T,x),
\end{align}
where in the last equality we used \eqref{FK-stopping}.
Then using the definition of $\widehat {\mathbf w}$ given in \eqref{v_w_relation_crit},
we get \eqref{v_feynman_crit}.

Notice that $\{B_t, t\geq 0;\Pi_x^{(y,\lambda^*)}\}$ is not a Brownian motion.
By the argument in the last paragraph of Subsection \ref{ss:mcBB}, we have that $\{y+\gamma'(\lambda^*)T(t)+h(B_{T(t)}), \, \Pi_x^{(y,\lambda^*)}\}$ is a Bessel-3 process starting from $y+h(x)$.
Define
\begin{equation}\label{def-f-St}
f(B_t) = \widehat {\mathbf w}(-t,B_t,y) e^{-\int_0^t \mathbf g(B_s)\mathbf w(-s,B_s)\mathrm{d}s}.
\end{equation}
By the  Markov property, for any $0<s<t$,
\begin{align}
\Pi_x^{(y,\lambda^*)}[f(B_{t\wedge\tau_z})| \F_{s}]= f(B_{s\wedge\tau_z}).
\end{align}
The proof of the display above is given in the Appendix, see Lemma \ref{local mart}.
This implies that $\{f(B_{t\wedge\tau_z}), \Pi_x^{(y,\lambda^*)}\}$ is a martingale. Since $\tau_z\rightarrow\infty$ as $z\downarrow 0$, $\{f(B_t), \Pi_x^{(y,\lambda^*)}\}$ is a local martingale.

\textbf{Step 3} Non-negative local martingales are
non-negative super-martingales and hence must converge.
Therefore $f(B_t)$, defined by \eqref{def-f-St},
converges $\Pi_x^{(y,\lambda^*)}$-a.s. as $t\to\infty$.
By Step 1,  $\mathbf w(t,x)$ decays exponentially fast uniformly in $t\in [0,\frac{1}{\nu^*}]$.
Combining this with the fact that
$B_t+\nu^*t = B_t + \gamma'(\lambda^*)t$ grows no slower than $t^{1/2-\epsilon}$, we have
$$\int_0^{\infty} \mathbf g(B_t) \mathbf w(-t,B_t) \mathrm{d}t < +\infty,\quad \Pi_x^{(y,\lambda^*)}\mbox{-a.s.}$$
Therefore, by \eqref{def-f-St}, the convergence of $f(B_t)$ implies that $\widehat {\mathbf w}(-t,B_t,y)$ converges $\Pi_x^{(y,\lambda^*)}$-almost surely to some limit, say $\xi_x$.

By \eqref{v_feynman_crit}, we have
\begin{equation}\label{v_feynman_crit_y}
 \widehat {\mathbf w}(T,x,y) = \Pi_x^{(y-\nu^*T,\lambda^*)} \left( \widehat {\mathbf w}(T-t\wedge\tau_z(y-\nu^*T), B_{t\wedge\tau_z(y-\nu^*T)}, y) e^{-\int_0^{t\wedge\tau_z(y-\nu^*T)} \mathbf g(B_s) \mathbf w(T-s,B_s) \mathrm{d}s } \right).
\end{equation}
By the same method, we can get
\begin{equation}
\{\widehat {\mathbf w}(T-t\wedge\tau_z(y-\nu^*T), B_{t\wedge\tau_z(y-\nu^*T)}, y) e^{-\int_0^{t\wedge\tau_z(y-\nu^*T)} \mathbf g(B_s) \mathbf w(T-s,B_s) \mathrm{d}s}, \, \Pi_x^{(y-\nu^*T,\lambda^*)} \}
\end{equation}
is a local martingale and $\widehat {\mathbf w}(T-t,B_t,y)$ converges $\Pi_x^{(y-\nu^*T,\lambda^*)}$-almost surely to some limit, say $\xi_x^T$.

Next we use a coupling method to prove that
there is a constant $\beta\geq 0$ such that
\begin{equation}\label{constant-forall-T}
\xi_x^T = \beta, \quad \Pi_x^{(y-\nu^*T,\lambda^*)}\mbox{-a.s.}\quad \forall T\geq 0.
\end{equation}
Similar to Step 2 in the proof of Theorem \ref{thrm_asym_super}, consider a process $\{(B_t^1,B_t^2), t\geq 0; \widetilde{\Pi}_x^{(y,T)} \}$
with $\{B_t^1, t\geq 0\}$ and $\{B_t^2, t\geq 0\}$ being independent, and
\begin{equation}
\left\{B_t^1, t\geq 0; \widetilde{\Pi}_x^{(y,T)}\right\} \overset{d}{=} \left\{B_t, t\geq 0; \Pi_x^{(y,\lambda^*)}\right\}, \quad \left\{B_t^2, t\geq 0; \widetilde{\Pi}_x^{(y,T)}\right\} \overset{d}{=} \left\{B_t, t\geq 0; \Pi_x^{(y-v^*T,\lambda^*)}\right\}.
\end{equation}
Define random curves
\begin{equation}
\mathcal{L}_k^1 = \left\{\left(-t+\frac{k}{\nu^*}, B_t^1+k\right): t\geq\frac{k-1}{\nu^*} \right\}, \quad k\in\N,
\end{equation}
and
\begin{equation}
\mathcal{L}_k^2 = \left\{\left(T-t+\frac{k}{\nu^*}, B_t^2+k\right): t\geq\frac{k-1}{\nu^*}+T \right\}, \quad k\in\N.
\end{equation}
Notice that all the curves start from the line $\{\frac{1}{\nu^*}\} \times \R$, and for each $i=1, 2$, if $\mathcal{L}_1^i$ is given, we can get all the curves $\mathcal{L}_k^i$ by translation. Using the fact that $y + h(B_{T(t)}) + \gamma'(\lambda^*)T(t)$ is a Bessel-3 process, $\gamma'(\lambda^*) = \nu^*$ and $|h(x) -x|$ is bounded, we have
\begin{equation}
\lim_{t\rightarrow\infty} B_t + \nu^*t = \infty \quad \mbox{$\Pi_x^{(y,\lambda^*)}$-a.s.}
\end{equation}
Now we show for $\widetilde{\Pi}_x^{(y,T)}$-almost surely all $\omega$, it holds that for any $k\in\N$, $\mathcal{L}_k^1$ and $\mathcal{L}_{k+1}^1$ intersect each other. It is also equivalent to show for any $k\in\N$, $\mathcal{L}_k^1$ and $\mathcal{L}_{k+1}^1$ intersect each other $\widetilde{\Pi}_x^{(y,T)}$-almost surely.
If there exists $t\geq \frac{k-1}{v^*}$ such that $B_{t+\frac{1}{\nu^*}}^1 + 1 = B_t^1 $, then
\begin{equation}
\left(-t+\frac{k}{\nu^*}, B_t^1+k\right) = \left(-\left(t+\frac{1}{\nu^*}\right)+\frac{k+1}{\nu^*}, B_{t+\frac{1}{\nu^*}}^1+k+1\right),
\end{equation}
which implies $\mathcal{L}_k^1$ and $\mathcal{L}_{k+1}^1$ intersect each other. Notice that
\begin{align}
B_{t+\frac{1}{\nu^*}}^1 + 1 = B_t^1 &\Longleftrightarrow B_{t+\frac{1}{\nu^*}}^1 + \nu^*t + 1 = B_t^1 + \nu^*t\\ &\Longleftrightarrow h(B_{t+\frac{1}{\nu^*}}^1) + \nu^*(t +\frac{1}{\nu^*}) = h(B_t^1) + \nu^*t\\
&\Longleftrightarrow \widehat R_{\langle M^1 \rangle_{t+\frac{1}{\nu^*}}} = \widehat R_{\langle M^1 \rangle_{t}},
\end{align}
where $\widehat R_t:=y+h(B_{T(t)}^1) + \nu^*T(t)$ is a standard Bessel-3 process started at $y+h(x)$.
By \eqref{quadratic}, we have
\begin{equation}
\langle M^1 \rangle_{t+\frac{1}{\nu^*}} - \langle M^1 \rangle_t \in \left[ \frac{c_1}{\nu^*}, \frac{c_2}{\nu^*} \right],
\end{equation}
where $\langle M^1 \rangle_t = \int_0^t (h'(B_s^1))^2 \mathrm{d}s$.
Put
\begin{equation}
l(t) = \widehat R_{\langle M^1 \rangle_{t+\frac{1}{\nu^*}}} -\widehat R_{\langle M^1 \rangle_{t}},
\end{equation}
then $l(t)$ is continuous $\widetilde{\Pi}_x^{(y,T)}$-almost surely.
Since $\widetilde{\Pi}_x^{(y,T)}\left(\lim_{t\to\infty}\widehat R_t =\infty\right)=1$, we have
\begin{equation}
\widetilde{\Pi}_x^{(y,T)} \left( \mbox{ for any }T>0, \exists\ t>T \mbox{ s.t. } l(t)>0 \right) = 1.
\end{equation}
To prove $\mathcal{L}_k^1$ and $\mathcal{L}_{k+1}^1$ intersect $\widetilde{\Pi}_x^{(y,T)}$-almost surely, it suffices to show that
\begin{equation}
\widetilde{\Pi}_x^{(y,T)} ( \mbox{ for any } T>0, \exists\ t>T \mbox{ s.t. } l(t)<0) = 1.
\end{equation}
Notice that
\begin{equation}
l(t) \leq \max_{s\in [\frac{c_1}{\nu^*},\frac{c_2}{\nu^*}]}\left( \widehat R_{s+\langle M^1 \rangle_{t}} - \widehat R_{\langle M^1 \rangle_{t}}\right).
\end{equation}
For simplicity, we assume $\nu^*=1$ in the remainder of this step. It suffices to show that
\begin{equation}\label{Bessel:e1}
\widetilde{\Pi}_x^{(y,T)} \left( \mbox{ for any }T>0, \exists\ t>T \mbox{ s.t. } \max_{s\in [c_1,c_2]}( \widehat R_{t+s} - \widehat R_t) < 0 \right) = 1.
\end{equation}
A classical result shows $\widehat R_t$ satisfies
\begin{equation}\label{Bessel:sde}
\mathrm{d} \widehat R_t = \mathrm{d} \widehat{B}_t + \frac{1}{\widehat R_t} \mathrm{d}t,
\end{equation}
where $(\widehat{B}_t; \widetilde{\Pi}_x^{(y,T)})$ is a standard Brownian motion.
By \eqref {Bessel:sde}, we have
\begin{align}
\widetilde{\Pi}_x^{(y,T)} \left(\max_{s\in [c_1,c_2]}( \widehat{B}_{t+s} - \widehat{B}_t )< -1 \right) &\geq \widetilde{\Pi}_x^{(y,T)} \left(\widehat{B}_{t+c_1} - \widehat{B}_t < -2,\, \max_{s\in [c_1,c_2]} \widehat{B}_{t+s} - \widehat{B}_{t+c_1} < 1 \right)\\
&= \Pi_0 (B_{c_1}<-2) \cdot \Pi_0 \left(\max_{s\in[0,c_2-c_1]} B_s < 1 \right) \geq C > 0,
\end{align}
where the constant $C$ does not depend on $t$. Hence
\begin{equation}
\sum_{j=0}^{\infty} \widetilde{\Pi}_x^{(y,T)} \left(\max_{s\in [c_1,c_2]} (\widehat{B}_{jc_2+s} - \widehat{B}_{jc_2}) < -1\right) = +\infty.
\end{equation}
By the second Borel-Cantelli lemma and the independent increments property of the Brownian motion, we have
\begin{equation}\label{Bessel:e2}
\widetilde{\Pi}_x^{(y,T)} \left(\max_{s\in [c_1,c_2]}( \widehat{B}_{jc_2+s} - \widehat{B}_{jc_2} )< -1, \mbox{ i.o. }\right) = 1.
\end{equation}
By $\eqref{Bessel:sde}$, it holds that
\begin{equation}
\widehat R_{t+s} - \widehat R_t = \widehat{B}_{t+s} - \widehat{B}_t + \int_t^{t+s} \frac{1}{\widehat R_r} \mathrm{d}r.
\end{equation}
If $\widehat R_{t+r}>c_2$ for $r\in [0,c_2]$ and $\widehat{B}_{t+s} - \widehat{B}_t < -1$ for $s\leq c_2$, then
\begin{equation}\label{Bessel:e3}
\widehat R_{t+s} -\widehat R_t <  -1 + \frac{s}{c_2} \leq 0.
\end{equation}
Since the Bessel-3 process is transient, we have
\begin{equation}\label{Bessel:transient}
\widetilde{\Pi}_x^{(y,T)}( \exists\ T>0, \mbox{ s.t. for any } t>T,  \widehat R_t>c_2 ) = 1.
\end{equation}
Combining \eqref{Bessel:e2}, \eqref{Bessel:e3} and \eqref{Bessel:transient}, we get \eqref{Bessel:e1}.
So $\widetilde{\Pi}_x^{(y,T)}$-almost surely,
it holds that for any $k\in\N$, $\mathcal{L}_k^1$ and $\mathcal{L}_{k+1}^1$ intersect each other. Using the same method, we can also prove
$\widetilde{\Pi}_x^{(y,T)}$-almost surely,
for any $k,j\in\N$, $\mathcal{L}_k^1$ and $\mathcal{L}_{k+j}^1$ intersect each other.

Consider
\begin{align}
\Omega_0 = &\bigcap_{k,j\in\mathbb{N}} \left\{\mbox{$\mathcal{L}_k^1$, $\mathcal{L}_{k+j}^1$ intersect}\right\} \bigcap\left\{\lim_{t\rightarrow\infty}  \widehat {\mathbf w}(-t, B_t^1,y),\, \lim_{t\rightarrow\infty} \widehat {\mathbf w}(T-t,B_t^2,y) \mbox{ exists} \right\}  \\
&\bigcap \left\{ \lim_{t\rightarrow\infty}B_t^i + \nu^*t = +\infty, \mbox{ and $B_t^i$ is continuous for $i=1,2$} \right\},
\end{align}
then $\widetilde{\Pi}_x^{(y,T)}(\Omega_0) = 1$. For any $\omega\in \Omega_0$ and $j\in\N$, we know $\mathcal{L}_j^2$ starts from the point $\left(\frac{1}{\nu^*}, B^2_{\frac{j-1}{\nu^*}+T}+j\right)$ and there exists $k\in\N$ such that
\begin{equation}
B_{\frac{k-1}{\nu^*}}^1(\omega) + k \leq B^2_{\frac{j-1}{\nu^*}+T}(\omega)+j \leq B_{\frac{k}{\nu^*}}^1(\omega) + k + 1.
\end{equation}
Hence the starting point of $\mathcal{L}_j^2(\omega)$ is between the starting points of $\mathcal{L}_k^1(\omega)$ and $\mathcal{L}_{k+1}^1(\omega)$. Since $\mathcal{L}_k^1(\omega)$ and $\mathcal{L}_{k+1}^1(\omega)$ intersect, we have $\mathcal{L}_j^2(\omega)$ must intersect either $\mathcal{L}_k^1(\omega)$ or $\mathcal{L}_{k+1}^1(\omega)$. We use $(s_1(\omega),x_1(\omega))$ to denote the intersection point. By \eqref{v_periodic_crit}, there exist $t_1^1(\omega)$, $t_1^2(\omega)$ satisfying
\begin{equation}
\widehat {\mathbf w}\left(-t_1^1(\omega), B_{t_1^1}^1(\omega), y\right) = \widehat {\mathbf w}\left(s_1(\omega),x_1(\omega),y\right) =  \widehat {\mathbf w}\left(T-t_1^2(\omega), B_{t_1^2}^2(\omega),y\right).
\end{equation}
Since $j$ is arbitrary, we can find $\{t_n^i(\omega):\, n\in\N, \, i=1,2\}$ by induction such that
\begin{equation}
\widehat {\mathbf w}\left(-t_n^1(\omega), B_{t_n^1}^1(\omega), y\right) =  \widehat {\mathbf w}\left(T-t_n^2(\omega), B_{t_n^2}^2(\omega),y\right)
\end{equation}
and satisfying
\begin{equation}
\lim_{n\rightarrow\infty} t_n^i(\omega) = \infty \quad \mbox{ for $i=1,2$.}
\end{equation}
Therefore, we have
\begin{equation}
\widetilde{\Pi}_x^{(y,T)} \left(\lim_{t\rightarrow\infty} \widehat {\mathbf w}(-t,B_t^1,y) = \lim_{t\rightarrow\infty} \widehat {\mathbf w}(T-t,B_t^2,y)\right) = 1.
\end{equation}
By the independence of $\{B_t^1, t\geq 0\}$ and $\{B_t^2, t\geq 0\}$, we get the limits must be the  same. So there is a constant $\beta\geq 0$ such that
\begin{equation}
\xi_x = \beta, \quad \Pi_x^{(y,\lambda^*)}\mbox{-a.s.} \quad \mbox{ and } \quad
\xi_x^T = \beta, \quad \Pi_x^{(y-v^*T,\lambda^*)}\mbox{-a.s.}
\end{equation}
Thus \eqref{constant-forall-T} is valid.

\textbf{Step 4}
In this step we prove $\beta>0$  by contradiction.
If $\beta=0$,
then $\widehat {\mathbf w}(-t,B_t,y) \rightarrow 0$ as $t\rightarrow\infty$.
Hence the positive local martingale
\begin{equation}
\widehat {\mathbf w}(-t,B_t,y)e^{-\int_0^t \mathbf g(B_s)\mathbf w(-s,B_s) \mathrm{d}s} \rightarrow 0, \quad \mbox{ $\Pi_x^{(y,\lambda^*)}$-a.s. as $t\rightarrow\infty$}.
\end{equation}
In Step 2 we proved that $\{f(B_{t\wedge\tau_z}), \Pi_x^{(y,\lambda^*)}\}$ is a martingale.
Thus if $y+h(x)>z>0$,  we have
\begin{align}
\widehat {\mathbf w}(0,x,y) &= \Pi_x^{(y,\lambda^*)} \left( \widehat {\mathbf w}(-t\wedge\tau_z, B_{t\wedge\tau_z}, y) e^{-\int_0^{t\wedge\tau_z} \mathbf g(B_s) \mathbf w(-s,B_s) \mathrm{d}s } \right)\\
&= \Pi_x^{(y,\lambda^*)} \left( \widehat {\mathbf w}(-\tau_z, B_{\tau_z}, y) e^{-\int_0^{\tau_z} \mathbf g(B_s) \mathbf w(-s,B_s) \mathrm{d}s } \mathbf{1}_{\{\tau_z<\infty\}} \right)\\
&= \Pi_x^{(y,\lambda^*)} \left( \frac{e^{\lambda^*B_{\tau_z}+\gamma(\lambda^*)\tau_z} \mathbf w(-\tau_z, B_{\tau_z})} {\psi( B_{\tau_z},\lambda^*)(y+\gamma'(\lambda^*)\tau_z + h(B_{\tau_z}))} e^{-\int_0^{\tau_z} \mathbf g(B_s) \mathbf w(-s,B_s) \mathrm{d}s } \mathbf{1}_{\{\tau_z<\infty\}} \right)\\
&= \Pi_x^{(y,\lambda^*)} \left( \frac{e^{\lambda^*B_{\tau_z}+\gamma(\lambda^*)\tau_z} \mathbf w(-\tau_z, B_{\tau_z})} {\psi( B_{\tau_z},\lambda^*)z} e^{-\int_0^{\tau_z} \mathbf g(B_s) \mathbf w(-s, B_s) \mathrm{d}s } \mathbf{1}_{\{\tau_z<\infty\}} \right).
\end{align}
Since $\{y+\gamma'(\lambda^*)T(t)+h(B_{T(t)}), \, \Pi_x^{(y,\lambda^*)}\}$ is a Bessel-3 process starting from $y+h(x)$, we have
\begin{equation}\label{tau-z<infty}
\Pi_x^{(y,\lambda^*)}(\tau_z<\infty) = \Pi_x^{(y,\lambda^*)}\left(\inf_{t\geq 0}\{ y+\gamma'(\lambda^*)T(t)+h(B_{T(t)})\} \leq z\right) = \frac{z}{y+h(x)}.
\end{equation}
By \eqref{v_w_relation_crit}, we have
\begin{align}
&\frac{e^{\lambda^*x}\mathbf w(0,x)}{\psi(x,\lambda^*)(y+h(x))}
= \widehat {\mathbf w}(0,x,y) \leq \Pi_x^{(y,\lambda^*)} \left( \frac{e^{\lambda^*B_{\tau_z}+\gamma(\lambda^*)\tau_z} \mathbf w(-\tau_z, B_{\tau_z})} {\psi( B_{\tau_z},\lambda^*)z}  \mathbf{1}_{\{\tau_z<\infty\}} \right)\\
&= \Pi_x^{(y,\lambda^*)} \left( \frac{e^{\lambda^*(z-y+\psi_{\lambda}(B_{\tau_z},\lambda^*)/ \psi(B_{\tau_z},\lambda^*))} \mathbf w(-\{\tau_z\}, B_{\tau_z}+\nu^*[\tau_z])} {\psi( B_{\tau_z},\lambda^*)z}  \mathbf{1}_{\{\tau_z<\infty\}} \right)\\
&\leq C_1 \Pi_x^{(y,\lambda^*)}(\tau_z<\infty)
\leq C_1 \frac{z}{y+h(x)},
\end{align}
where we used the facts that  $\psi_{\lambda}/\psi$ is bounded and that $\mathbf w(t,x)$ is bounded in
$(t,x)\in [-\frac{1}{\nu^*},0]\times [-C_2,C_2]$.
Here $C_1,C_2$ are constants only depending only on $y,z,\lambda^*$. Hence we get that for $x>h^{-1}(z-y)$, $e^{\lambda^*x}\mathbf w(0,x)/\psi(x,\lambda^*)$ is bounded. Combining this with the fact $e^{\lambda^*x}\mathbf w(0,x)/\psi(x,\lambda^*) \rightarrow 0$ as $x\rightarrow-\infty$, we have $e^{\lambda^*x}\mathbf w(0,x)/\psi(x,\lambda^*)$ is bounded on  $\R$.
Similarly, we can prove that
$e^{\lambda^*x-\gamma(\lambda^*)t} \mathbf w(t,x)/\psi(x,\lambda^*)$ is bounded on  $[0,\frac{1}{\nu^*}]\times\R$.

By Step 1 of the proof of Theorem \ref{thrm_asym_super}, we know that
\begin{equation}
\left\{\frac{e^{\lambda^*Y_t+\gamma(\lambda^*)t}\mathbf w(-t,Y_t)}{\psi(Y_t,\lambda^*)} e^{-\int_0^t \mathbf g(Y_s)\mathbf w(-s,Y_s) \mathrm{d}s }, \Pi_x^{\lambda^*}\right\}_{t\geq 0}
\end{equation}
is a martingale, where $\{Y_t, \Pi_x^{\lambda^*}\}$ is a diffusion with infinitesimal generator
\begin{equation}
(\mathcal{A}f)(x) = \frac{1}{2} \frac{\partial^2 f(x)}{\partial x^2} + \left( \frac{\psi_x(x,\lambda^*)}{\psi(x,\lambda^*)} - \lambda^* \right) \frac{\partial f(x)}{\partial x}.
\end{equation}
So $\left\{\frac{e^{\lambda^*Y_t+\gamma(\lambda^*)t} \mathbf w(-t,Y_t)}{\psi(Y_t,\lambda^*)} , \Pi_x^{\lambda^*}\right\}$ is a positive submartingale that is bounded and hence must converge.
Using the argument similar to that in Step 2 in the proof of Theorem \ref{thrm_asym_super}, we have the limit of $\frac{e^{\lambda^*Y_t+\gamma(\lambda^*)t} \mathbf w(-t,Y_t)}{\psi(Y_t,\lambda^*)}$ is a constant.
Since $$\mathbf w(-t,Y_t) \leq 1, \; \liminf_{t\rightarrow\infty} (Y_t+\nu^*t) = -\infty \mbox{ and } \lim_{x\rightarrow-\infty}  e^{\lambda^*x}/\psi(x,\lambda^*) = 0,$$
we must have the constant is $0$. This implies $\mathbf w(0,x) \equiv 0$, which contradicts the definition of pulsating travelling waves.
Therefore, we have $\beta>0$.

\textbf{Step 5}
Fix $y\in\R$.
First we will show that $\widehat {\mathbf w}(t,x,y)$ is bounded for $(t,x)\in [0,\frac{1}{\nu^*}]\times [h^{-1}(z-y+\nu^*t),\infty)$, which implies that $(t,x)$ satisfies $y-\nu^*t+h(x)\geq z>0$.
Recall that $\phi := \phi(x,\lambda^*) = e^{-\lambda^*x} \psi(x,\lambda^*)$. By \eqref{v_w_relation_crit}, we can rewrite $\widehat {\mathbf w}$ as
\begin{equation}
\widehat {\mathbf w}(t,x,y) = \frac{e^{-\gamma(\lambda^*)t}\mathbf w(t,x)} {\phi(x,\lambda^*)\left(y-\gamma'(\lambda^*)t\right) - \phi_{\lambda}(x,\lambda^*)}.
\end{equation}
By \eqref{phi-equ} and \eqref{phi_lambda_eq}, a direct calculation yields
\begin{equation}\label{v_equation_crit}
\frac{\partial \widehat {\mathbf w}}{\partial t} = \frac{1}{2} \frac{\partial^2 \widehat {\mathbf w}}{\partial x^2} + \frac{(y-\gamma'(\lambda^*))\phi_x - \phi_{\lambda x}} {(y-\gamma'(\lambda^*))\phi - \phi_{\lambda}} \frac{\partial \widehat {\mathbf w}}{\partial x} -
\mathbf g \mathbf w \widehat {\mathbf w}.
\end{equation}

Similar to Step 3 in the proof of Theorem \ref{thrm_asym_super}, fix $\omega\in\Omega_0$, then for any $x_0$ large enough, there exist $k,j\in\N$ such that
$(t,x_0)$ is located between the curves $\mathcal{L}_j^1(\omega)$ and $\mathcal{L}_k^1(\omega)$ for any $t\in [0,\frac{1}{\nu^*}]$.
Since $\mathcal{L}_j^1(\omega)$ and $\mathcal{L}_k^1(\omega)$ must intersect each other, it follows from the maximum principle that $\widehat {\mathbf w}(t,x_0,y)$ is bounded by the maximum on boundary $\mathcal{L}_j^1(\omega)$ and $\mathcal{L}_k^1(\omega)$.
We know that $\widehat {\mathbf w}(t,x,y)$  along  $(t,x) \in\mathcal{L}_k^1(\omega)$ converges as $t\to\infty$, and then  $\widehat {\mathbf w}(t,x,y)$  is bounded  on $\bigcup_{j\geq 1}\mathcal{L}_j^1(\omega)$.
Since $\widehat {\mathbf w}(t,x,y)$ is continuous and $\widehat {\mathbf w}(t,h^{-1}(z-y+\nu^*t),y)$ is bounded, we have $\widehat {\mathbf w}(t,x_0,y)$ is bounded when $x_0$ is small.
Therefore,
$\widehat {\mathbf w}(t,x,y)$ is bounded in $[0,\frac{1}{\nu^*}]\times [h^{-1}(z-y+\nu^*t),\infty)$. We denote the bound by  $K_z$.

Next, we will show that
\begin{equation}\label{v_converge_crit}
\widehat {\mathbf w}(t,x,y) \rightarrow \beta \quad
 \mbox{ as }x\rightarrow\infty\mbox{ uniformly in } t\in \left[0,\frac{1}{\nu^*}\right].
\end{equation}
Recall that
\begin{equation}
 \widehat {\mathbf w}(T,x,y) = \Pi_x^{(y-\nu^*T,\lambda^*)} \left( \widehat {\mathbf w}(T-t\wedge\tau_z(y-\nu^*T), B_{t\wedge\tau_z(y-\nu^*T)}, y) e^{-\int_0^{t\wedge\tau_z(y-\nu^*T)} \mathbf g(B_s) \mathbf w(T-s,B_s) \mathrm{d}s } \right).
\end{equation}
Letting  $t\rightarrow\infty$, we have for $T\in [0,\frac{1}{\nu^*}]$,
\begin{align}
\widehat {\mathbf w}(T,x,y) &\leq  \Pi_x^{(y-\nu^*T, \lambda^*)}\left(\beta \mathbf{1}_{\{\tau_z(y-\nu^*T)=\infty\}} + K_z \mathbf{1}_{\{\tau_z(y-\nu^*T)<\infty\}}\right)\\
&= \beta \left(1-\frac{z}{h(x)+y-\nu^*T}\right) + K_z \frac{z}{h(x)+y-\nu^*T},
\end{align}
where we used \eqref{tau-z<infty}.
Therefore, we have
\begin{equation}\label{v_limsup}
\limsup_{x\rightarrow\infty} \widehat {\mathbf w}(T,x,y) \leq \beta \quad
\mbox{ uniformly for } T\in \left[0,\frac{1}{\nu^*}\right].
\end{equation}
On the other hand, we have the following estimate:
\begin{equation}
\widehat {\mathbf w}(T,x,y) \geq \Pi_x^{(y-\nu^*T, \lambda^*)} \left(\beta \mathbf{1}_{\{\tau_z(y-\nu^*T)=\infty\}} e^{-\int_0^{\infty} \mathbf g(B_t) \mathbf w(T-t,B_t) \mathrm{d}t}\right).
\end{equation}
By \eqref{v_w_relation_crit}, we get
\begin{align}
\mathbf g(B_t)\mathbf w(T-t,B_t) &\leq
\|\mathbf g\|_\infty z K_z \max_{x\in[0,1]} \psi(x,\lambda^*) e^{-\lambda^*B_t+\gamma(\lambda^*)(T-t)}\\
&\leq C e^{-\lambda^*(y-\nu^*T + h(B_t)+\nu^*t)}
= C e^{-\lambda^* \widehat R^T_{\langle M \rangle_t}},
\end{align}
where $\widehat R^T_t = y-\nu^*T + h(B_{T(t)})+\nu^*T(t)$
is a Bessel-3 process started at $y-\nu^*T+h(x)$,
$T(t) = \inf\{s>0:\langle M \rangle_s>t \}$,
and the constant $C$ does not depend on $T$. Thus we have
\begin{align}
\widehat {\mathbf w}(T,x,y) &\geq \Pi_x^{(y-\nu^*T, \lambda^*)}\left(\beta \mathbf{1}_{\{\tau_z(y-\nu^*T)=\infty\}} e^{-\int_0^{\infty} C e^{-\lambda^* \widehat R^T_{\langle M \rangle_t}} \mathrm{d}t}\right)\\
&\geq \Pi_x^{(y-\nu^*T, \lambda^*)}\left(\beta \mathbf{1}_{\{\tau_{z+\nu^*T}=\infty\}} e^{-\int_0^{\infty} C e^{-\lambda^* \widehat R^T_t} \mathrm{d}t}\right),
\end{align}
where $\tau_z(y-\nu^*T) = \tau_{z+\nu^*T}$.
Define the stopping time
\begin{equation}
\sigma_x(\widehat R^T) := \inf\left\{t\geq 0: \widehat R^T_t = y-\nu^*T+h(x) \right\},
\end{equation}
and the function
\begin{equation}
f(x) := \Pi_{x}^{(y-\nu^*T, \lambda^*)}\left(\beta \mathbf{1}_{\{\tau_{z+\nu^*T}=\infty\}} e^{-\int_0^{\infty} C e^{-\lambda^* \widehat R^T_t} \mathrm{d}t}\right).
\end{equation}
By the Markov property, for $x_1<x_2$ we have
\begin{align}
f(x_1) = \Pi_{x_1}^{(y-\nu^*T, \lambda^*)}\left[ \mathbf{1}_{\{\sigma_{x_2}(\widehat R^T) < \tau_{z+\nu^*T}\}} e^{-\int_0^{\sigma_{x_2}(\widehat R^T)} C e^{-\lambda^* \widehat R^T_t} \mathrm{d}t} f(x_2)\right] \leq f(x_2),
\end{align}
that is, $f(x)$ is increasing. Put
\begin{equation}
f(\infty) := \lim_{x\rightarrow\infty} f(x).
\end{equation}
Since $\beta>0$ and $\int_0^{\infty} e^{-\lambda^*\widehat R^T_t} \mathrm{d}t < \infty$, $\Pi_{x}^{(y-\nu^*T, \lambda^*)}$-almost surely,
we have $f(x) > 0$. Moreover,
\begin{align}
\beta f(x) &= \lim_{n\rightarrow\infty}\Pi_{x}^{(y-\nu^*T, \lambda^*)}\left[\beta \mathbf{1}_{\{\sigma_{x+n}(\widehat R^T) < \tau_{z+\nu^*T}\}} e^{-\int_0^{\sigma_{x+n}(\widehat R^T)} C e^{-\lambda^* \widehat R^T_t} \mathrm{d}t} f(x+n)\right] \\
&= \Pi_{x}^{(y-\nu^*T, \lambda^*)}\left[ \lim_{n\rightarrow\infty} \beta \mathbf{1}_{\{\sigma_{x+n}(\widehat R^T) < \tau_{z+\nu^*T}\}} e^{-\int_0^{\sigma_{x+n}(\widehat R^T)} C e^{-\lambda^* \widehat R^T_t} \mathrm{d}t} \right] \lim_{n\rightarrow\infty} f(x+n)\\
&= \Pi_{x}^{(y-\nu^*T, \lambda^*)}\left[  \beta \mathbf{1}_{\{\tau_{z+\nu^*T}=\infty\}} e^{-\int_0^{\infty} C e^{-\lambda^* \widehat R^T_t} \mathrm{d}t} \right] f(\infty) = f(x) f(\infty).
\end{align}
Therefore, $f(\infty) = \beta$. Combining this with $\widehat {\mathbf w}(T,x,y) \geq f(x)$, we have
\begin{equation}\label{v_liminf}
\liminf_{x\rightarrow\infty} \widehat {\mathbf w}(T,x,y) \geq \beta \quad
\mbox{ uniformly for }T\in \left[0,\frac{1}{\nu^*}\right].
\end{equation}
By \eqref{v_limsup} and \eqref{v_liminf}, we get \eqref{v_converge_crit}.

\textbf{Step 6} Since $h(x) = x - \psi_{\lambda}(x,\lambda^*)/\psi(x,\lambda^*)$ and $\psi_{\lambda}/\psi$ is bounded, we have
\begin{align}
\lim_{x\rightarrow\infty} \widehat {\mathbf w} (t,x,y) = \lim_{x\rightarrow\infty} \frac{e^{\lambda^*x-\gamma(\lambda^*)t}\mathbf w(t,x)} {\psi(x,\lambda^*)x}.
\end{align}
Thus
\begin{equation}
\frac{e^{\lambda^*z-\gamma(\lambda^*)t}\mathbf w(t,z)} {\psi(z,\lambda^*)z} \rightarrow \beta \quad
\mbox{as }z\rightarrow\infty\mbox{ uniformly in }t\in \left[0,\frac{1}{\nu^*}\right].
\end{equation}
Using an argument similar to that in Step 3 of the proof of Theorem \ref{thrm_asym_super}, we have
\begin{equation}
\frac{e^{\lambda^* x}\mathbf w\left(\frac{y-x}{\nu^*}, y\right)}{ \psi(y,\lambda^*)(-\lfloor y-x \rfloor + y)} \rightarrow \beta \quad
\mbox{as }x\rightarrow\infty\mbox{ uniformly in }y\in [0,1].
\end{equation}
By $|- \lfloor y-x \rfloor +y-x|\leq 1$, we have
\begin{equation}
\frac{e^{\lambda^* x}\mathbf w\left(\frac{y-x}{\nu^*}, y\right)}{ x\psi(y,\lambda^*)} \rightarrow \beta \quad
\mbox{as }x\rightarrow\infty\mbox{ uniformly in }y\in [0,1].
\end{equation}

\textbf{Step 7} For the general branching mechanism, $\mathbf w=1-\mathbf u$ satisfies
\begin{equation}
\frac{\partial \mathbf w}{\partial t} = \frac{1}{2} \frac{\partial^2 \mathbf w}{\partial x^2} +
\mathbf g\cdot(1-\mathbf w-\mathbf f(1-\mathbf w)).
\end{equation}
Recall that $A(w) = m - \frac{1-w-\mathbf f(1-w)}{w}$.
As in Step 3, it suffices to show that
\begin{equation}
\int_0^{\infty} \mathbf g(B_t) \mathbf w(-t,B_t) \mathrm{d}t < \infty,\quad \Pi_x^{(y,\lambda^*)}\mbox{-a.s.}
\end{equation}
Since $B_t+\nu^*t$ behaves like $\sqrt{t}$ and $\mathbf w(t,x)$ decays exponentially fast, we have
\begin{align}
\int_0^{\infty} A(e^{-c\sqrt{t}}) \mathrm{d}t < \infty \, \mbox{ for some $c>0$} \Longrightarrow\int_0^{\infty} \mathbf g(B_t) A(\mathbf w)(-t,B_t) \mathrm{d}t < \infty.
\end{align}
Set $s = e^{-c\sqrt{t}}$, then
\begin{equation}
\int_0^{\infty} A(e^{-c\sqrt{t}}) \mathrm{d}t < \infty \Longleftrightarrow \int_0^1 A(s) \frac{|\log s|}{s} \mathrm{d}s < \infty.
\end{equation}
By \cite[Theorem 2]{CR90}, it holds that for $a>1$,
\begin{equation}\label{A_LlogL2}
\int_0^1 A(s) \frac{|\log s|^a}{s} \mathrm{d}s < \infty  \Longleftrightarrow \mathbf E(L(\log^+L)^{1+a})<\infty.
\end{equation}
Actually the proof of \cite[Theorem 2]{CR90} also works for  $a=1$.
So if $\mathbf E(L(\log^+L)^2)<\infty$, there exists $\beta>0$ such that
\begin{equation}
1-\mathbf u\left(\frac{y-x}{\nu^*}, y\right) \sim \beta xe^{-\lambda^* x}\psi(y,\lambda^*) \text{ as } x\rightarrow +\infty \text{ uniformly in } y\in [0,1].
\end{equation}
This completes the proof.
\end{proof}

\begin{remark}
The asymptotic behavior of pulsating travelling waves has been studied analytically in \cite{H08}. They considered 	the following more general equation:
\begin{equation}
u_t - \nabla \cdot (A(z) \nabla u) + q(z) \cdot \nabla u = f(z,u), \quad z\in \overline{\Omega},
\end{equation}
where $\Omega\subset\R^N$ is an unbounded domain, $A(z), q(z), f(z,u)$ are periodic in some sense. In our case,
\begin{equation}
\Omega = \R, \; A(z) \equiv \frac{1}{2}, \; q(z) \equiv 0, \mbox{ and } f(z,u) = \mathbf{g}(z)(1-u-\mathbf f(1-u)).
\end{equation}
\cite[Theorem 1.3]{H08}, the main result of the paper,
is similar to our Theorems \ref{thrm_asym_super} and \ref{thrm_asym_crit}, where $\phi(t,x)$ in \cite{H08} corresponds to $1-\mathbf u \left( \frac{t+x}{\nu}, x\right)$ in this paper, and $p^-(x,y), p^+(x,y)$ correspond to $0,1$ respectively.

\cite[Theorem 1.3]{H08} was proved under the assumptions \cite[(1.4), (1.7) and (1.8)]{H08}.
In our setup,
\cite[(1.4)]{H08} is equivalent to $\gamma(0) > 0$, which is  true under our assumptions.
\cite[(1.8)]{H08} is equivalent to
\begin{equation}
f(x,s) = \mathbf g(x)(1-s - \mathbf f(1-s)) \leq m \mathbf g(x) s \quad \mbox{ for } s\in [0,1],
\end{equation}
which is true for any generating function $\mathbf f$.
\cite[(1.7)]{H08} says  that there exist $\alpha>0$ and $\gamma>0$ such that the map $(x,s) \mapsto \mathbf g(x)(\mathbf f'(1-s)-1)$ belongs to $C^{0,\alpha}(\R\times [0,\gamma])$, which is equivalent to $g\in C^{0,\alpha}(\R)$ and $\mathbf f'(1-s)\in C^{0,\alpha}([0,\gamma])$.
We claim that
	\begin{equation}\label{Holder_LlogL}
	\mathbf f'(1-s) \in C^{\alpha}([0,\gamma]) \Longrightarrow \forall p \geq 1, \; \mathbf E(L(\log^+L)^{p}) < \infty.
	\end{equation}
Thus
the condition $\mathbf f'(1-s)\in C^{0,\alpha}([0,\gamma])$
is stronger than the condition $\mathbf E(L\log^+L)<\infty$ in the supercritical case and $\mathbf E(L(\log^+L)^2)<\infty$ in the critical case.
	Now we prove the claim. Notice that
	\begin{align}
	A(w) = m - \frac{1-w-\mathbf f(1-w)}{w} = m+1 - \frac{1-f(1-w)}{w} = m+1 - f'(1-\theta w),
	\end{align}
	where $\theta\in[0,1]$ and the last equality follows from the mean value theorem. Since  $f'(1) = m+1$ and $\mathbf f'(1-s) \in C^{\alpha}([0,\gamma])$, we have
	\begin{equation}
	A(w) = |m+1 - f'(1-\theta w)|
\leq C (\theta w)^{\alpha}  \leq C w^{\alpha},\quad \forall w\leq \gamma,
	\end{equation}
for some constant $C$.
	Therefore, for any constant $c>0$,
	\begin{equation}
	\int_0^{\infty} A(e^{-ct^{1/p}}) \mathrm{d}t \leq \int_0^{\infty} Ce^{-c\alpha t^{1/p}} \mathrm{d}t < \infty.
	\end{equation}
Using the substitution $s = e^{-ct^{1/p}}$,  we get
    \begin{equation}
    \int_0^1 A(s) \frac{|\log s|^{p-1}}{s} \mathrm{d}s < \infty.
    \end{equation}
	By \eqref{A_LlogL2}, this implies $\mathbf E(L(\log^+L)^p)<\infty$.
\end{remark}

\section{Proof of Theorem \ref{thrm3}}
The uniqueness of the pulsating travelling wave was proved analytically in \cite[Theorem 1.1]{HR}.
In this section, we will use probabilistic methods to prove the uniqueness in the supercritical case $|\nu|>\nu^*$ and critical case $|\nu|=\nu^*$.
\subsection{Martingales on stopping lines}\label{s:5.3}
First, we introduce the space of Galton-Watson trees. Let $\mathbb{T}$ be the space of Galton-Watson trees. A Galton-Watson tree $\tau \in \mathbb{T}$ is a point in the space of possible Ulam-Harris labels
$$
\Omega=\emptyset \cup \bigcup_{n \in \mathbb{N}}(\mathbb{N})^{n},
$$
where $\mathbb{N}=\{1,2,3, \ldots\}$ such that

\begin{enumerate}[(i)]
	\item $\emptyset \in \tau$ (the ancestor);
	\item if $u, v \in \Omega$, $uv \in \tau$ implies $u \in \tau$;
	\item for all $u \in \tau,$ there exists $A_u \in\{0,1,2, \ldots\}$
	such that for $j \in \mathbb{N}$,  $j \in \tau$ if and only if $1 \leq j \leq 1+A_u$.
\end{enumerate}
(Here $1+A_u$ is the number of offspring of $u$,
and $A_u$ has the same distribution as $L$.)

Each particle $u\in\tau$ has a mark $(\eta_u, B_u)\in \R^+\times C(\R^+,\R)$, where $\eta_u$ is the lifetime of $u$ and $B_u$ is the motion of $u$ relative to its birth position. Then the birth time of $u$ can be written as $b_u = \sum_{v<u} \eta_v$, the death time of $u$ is $d_u = \sum_{v\leq u}\eta_v$ and the position of $u$ at time $t$ is given by $X_u(t) = \sum_{v<u} B_v(\eta_v)+B_u(t-b_u)$, where $v<u$ denotes that $u$ is a descendant of $v$.

Now, on the space-time half plane $\{(y,t): y\in\R, t\in\R^+ \}$, consider the barrier $\Gamma^{(x,\nu)}$ described by the line $y+\nu t=x$ for $x>0$
and $\nu\geq \nu^*$.
When a particle hits this barrier,
it is stopped immediately.
Let $C(x,\nu)$ denote the random collection of particles stopped at the barrier, which is known as a stopping line.

By \cite[Theorem 1.1]{RSYa}, $W_t(\lambda^*) \rightarrow 0$ $\P_x$-almost surely, so we have
\begin{equation}
e^{-\lambda^*(m_t+\nu^*t)}
\min_{x\in[0,1]} \psi(x,\lambda^*) \leq W_t(\lambda) \rightarrow 0,
\end{equation}
where $m_t = \min\{X_u(t): u\in N_t\}$.
This yields
\begin{equation*}
\lim_{t\rightarrow\infty} (m_t + \nu^*t) = +\infty.
\end{equation*}
Therefore, all lines of descent from the ancestor will hit $\Gamma^{(x,\nu)}$ with probability one for all $x>x_0$, where $x_0$ is the position of the ancestor at time $t=0$. Similar to the argument in \cite{Ky} for BBM, we have $\lim_{x\rightarrow\infty} \inf\{|u|:u\in C(x,\nu) \} = \infty$ where $|u|$ is the generation of particle $u$.

For any $u\in C(x,\nu)$, let $\sigma_u$ denote the time at which the particle $u$ hits the barrier $\Gamma^{(x,\nu)}$.
Let $\F_{C(x,\nu)}$ be the $\sigma$-field generated by
\begin{align}
\left\{\begin{array}{l}
(w, A_w, \eta_{w},\{B_{w}(s): s \in[0, \eta_w]\}: \exists u\in C(x,\nu),\, s.t. \, w<u) \mbox{ and } \\
(u,\{B_{u}(s): s \in[0, \sigma_u-b_u]\}: u \in C(x,\nu))
\end{array}\right\}.
\end{align}

Using travelling wave solutions of the KPP equation, Chauvin \cite{Chauvin91} exhibited an intrinsic class of martingales.
An argument similar to the one used in \cite{Chauvin91}
gives the analogous martingales for BBMPE.

\begin{thrm}\label{thrm_mart_prod}
	Suppose that $\nu>\nu^*$ and that $\mathbf u(t,x)$ is a pulsating travelling wave with speed $\nu$. Define
	\begin{equation}\label{mart_prod}
	M_{x}(\nu) :=
	\prod_{u\in  C(x,\nu)} \mathbf u(-\sigma_u, X_u(\sigma_u)).
	\end{equation}
	Then, for any $z\in \R$, $\{M_{x}(\nu): x\ge z\}$ is a $\P_z$-martingale with respect to $\{\F_{C(x,\nu)}, x\geq z\}$,  has expectation $u(0,z)$, and   converges $\P_z$-almost surely and in $L^1(\P_z)$.
\end{thrm}

We prove this theorem via several lemmas.

\begin{lemma}\label{lemma_prod_mean}
	Consider BBMPE starting from $x$ and with branching rate function $\mathbf g$.
	Let $\sigma$ be the first fission time and $1+A$ denote the number of offspring of the initial particle.  Let $\mathbf f(s) = \mathbf E s^{1+A}$ and $\mathbf P(A=k)=p_k$.
	Then
	\begin{equation}\label{prod_mean}
	\P_x\left(\mathbf{1}_{\{\sigma>t\}} \mathbf u(-t,B_t) + \mathbf{1}_{\{\sigma\leq t\}} \mathbf u^{1+A}(-\sigma,B_{\sigma}) \right) = \mathbf u(0,x).
	\end{equation}
\end{lemma}
\begin{proof} Note that
	\begin{equation*}
	\P_x \left(\sigma>t\;|\; \{B_s:s\leq t \}\right) = e^{-\int_0^t \mathbf g(B_s) \mathrm{d}s}.
	\end{equation*}
	Put
	\begin{equation}
	f(t,B_t) = e^{-\int_0^t \mathbf g(B_s) \mathrm{d}s} \mathbf u(-t,B_t) + \int_0^t \mathbf g(B_s) e^{-\int_0^s \mathbf g(B_r) \mathrm{d}r} \sum_{k} p_k \mathbf u^{k+1} (-s,B_s) \mathrm{d}s.
	\end{equation}
	A standard computation using It\^{o}'s formula shows that
	\begin{align*}
	&\mathrm{d}f(t,B_t) \\
	=\; & e^{-\int_0^t \mathbf g(B_s) \mathrm{d}s} \frac{\partial \mathbf u(-t,B_t) }{\partial x} \mathrm{\mathrm{d}}B_t + \frac{1}{2} e^{-\int_0^t \mathbf g(B_s) \mathrm{d}s} \frac{\partial^2 \mathbf u(-t,B_t) }{\partial x^2} \mathrm{d}t
	-\mathbf g(B_t) e^{-\int_0^t \mathbf g(B_s) \mathrm{d}s} \mathbf u(-t,B_t) \mathrm{d}t\\
	& - e^{-\int_0^t \mathbf g(B_s) \mathrm{d}s} \frac{\partial \mathbf u(-t,B_t) }{\partial t} \mathrm{d}t + \mathbf g(B_t) e^{-\int_0^t \mathbf g(B_s) \mathrm{d}s} \sum_{k} p_k \mathbf u^{k+1} (-t,B_t) \mathrm{d}t\\
	= \; & e^{-\int_0^t \mathbf g(B_s) \mathrm{d}s} \frac{\partial \mathbf u(-t,B_t) }{\partial x} \mathrm{\mathrm{d}}B_t\\
	&+ e^{-\int_0^t \mathbf g(B_s) \mathrm{d}s} \left(\frac{1}{2} \frac{\partial^2 \mathbf u(-t,B_t) }{\partial x^2} - \mathbf g(B_t)\mathbf u(-t,B_t) - \frac{\partial \mathbf u(-t,B_t) }{\partial t} + \mathbf g(B_t) \mathbf f(\mathbf u(-t,B_t)) \right)\mathrm{d}t\\
	=\; & e^{-\int_0^t \mathbf g(B_s) \mathrm{d}s} \frac{\partial \mathbf u(-t,B_t) }{\partial x} \mathrm{\mathrm{d}}B_t.
	\end{align*}
	Hence $f(t,B_t)$ is a martingale and
	$$
	\Pi_x f(t,B_t) = \Pi_x f(0,B_0) = \mathbf u(0,x).
	$$
	Note that
	\begin{align}
	\P_x (\mathbf{1}_{\{\sigma>t\}} \mathbf u(-t,B_t))= &\P_x\left( \P_x\left[\mathbf{1}_{\{\sigma>t\}} \mathbf u(-t,B_t) \;|\; \{B_s: s\leq t\} \right]  \right)\\
	= &\P_x \left(\mathbf u(-t,B_t) \P_x \left[\mathbf{1}_{\{\sigma>t\}} \;|\; \{B_s: s\leq t\} \right] \right)\\
	= &\Pi_x \left( e^{-\int_0^t \mathbf g(B_s) \mathrm{d}s} \mathbf u(-t,B_t) \right).
	\end{align}
	Similarly, we have
	\begin{align*}
	\P_x \left(\mathbf{1}_{\{\sigma\leq t\}} \mathbf u^{A+1} (-\sigma,B_{\sigma}) \right) &= \P_x\left( \P_x \left[\mathbf{1}_{\{\sigma\leq t\}} \mathbf u^{A+1}(-\sigma, B_{\sigma}) \;|\; \{B_s: s\leq t\} \right]  \right)\\
	&= \Pi_x \left( \int_0^t  \mathbf g(B_s) e^{-\int_0^s \mathbf g(B_r) \mathrm{d}r} \sum_k p_k \mathbf u^{k+1}(-s,B_s) \mathrm{d}s \right).
	\end{align*}
	Therefore,
	\begin{equation*}
	\P_x\left(\mathbf{1}_{\{\sigma>t\}} \mathbf u(-t,B_t) + \mathbf{1}_{\{\sigma\leq t\}} \mathbf u^{1+A}(-\sigma,B_{\sigma})\right)
	= \Pi_x f(t,B_t) = \mathbf u(0,x).
	\end{equation*}
\end{proof}

Recall that $\Omega=\emptyset \cup \bigcup_{n \in \mathbb{N}}(\mathbb{N})^{n}$ and that, for $u\in \Omega$, $\sigma_u$ denotes the time at which $u$ hits the barrier $\Gamma^{(x,\nu)}$.
For a fixed stopping line $C(x,\nu)$, we define
\begin{align}
L_{\tau} &:= \{u\in\Omega: \; b_u\le \sigma_u < d_u \} = C(x,\nu),\\
D_{\tau} &:= \{u\in\Omega: \; \exists v\in\Omega, \; v<u, \; v\neq u, \;v\in L_{\tau} \},\\
A_{\tau}^{(n)} &:= \{u\in\Omega: \; |u|=n, \; u\notin D_{\tau}, \; u\notin L_{\tau} \}.
\end{align}
In other words, $D_{\tau}$ is the set of strict descendants of the stopping line and $A_{\tau}^{(n)}$ is the set of the $n$th generation particles $u$ such that neither $u$ nor the ancestors of $u$ hit the barrier $\Gamma^{(x,v)}$.
Let
\begin{equation*}
\H_n=
\sigma\left(\{(u,A_u,\eta_u,\left\{B_u:u\in[0,\eta_u] \right\}  )\}: |u|\leq n-1  \right).
\end{equation*}
Recall that $M_{x}(v)$ is defined by \eqref{mart_prod}.
\begin{lemma}\label{lemma_M_tau}
	For any $z<x$,
	\begin{equation*}
	\E_z (M_{x}(\nu)) =\mathbf u(0,z).
	\end{equation*}
\end{lemma}
\begin{proof} For $n\in\N$, similar to \cite{Chauvin91}, we introduce the following approximation of $M_{x}(\nu)$,
	\begin{equation}\label{mart_Mn}
	M^{(n)} = \prod_{u\in L_{\tau}, |u|\leq n} \mathbf u(-\sigma_u, X_u(\sigma_u)) \prod_{u\in A_{\tau}^{(n)}} \mathbf u^{1+A_u} (-d_u,X_u(d_u)).
	\end{equation}
	It is easy to see that $M^{(n)}\in \H_{n+1}, n\geq 0$.
	We first prove that $\{M^{(n)}, n\geq 0\}$ is an $\H_{n+1}$-martingale.
	\begin{align}\label{mart_Mn+1}
	\E_x (M^{(n+1)} \;|\; \H_{n+1}) &= \prod_{u\in L_{\tau}, |u|\leq n} \mathbf u(-\sigma_u, X_u(\sigma_u))\notag \\
	& \times \E_x \bigg{(} \prod_{u\in L_{\tau}, |u|=n+1} \mathbf u(-\sigma_u, X_u(\sigma_u)) \prod_{u\in A_{\tau}^{(n+1)}} \mathbf u^{1+A_u}(-d_u,X_u(d_u)) \;|\; \H_{n+1} \bigg{)}.
	\end{align}
	Note that
	\begin{equation*}
	\{u: u\in L_{\tau},\;|u| = n+1 \} \cup A_{\tau}^{(n+1)} = \{u: |u| = n+1,\; u\notin D_{\tau} \}.
	\end{equation*}
	Consider any particle $u$ such that $|u|=n+1$ and $u\notin D_{\tau}$. If $u$ is in $L_{\tau}$, it occurs in the second product in \eqref{mart_Mn+1};
	if not, it occurs in the third one. By the Markov property and the branching property, we have
	\begin{align*}
	&\;\E_x \bigg{(} \prod_{u\in L_{\tau}, |u|=n+1} \mathbf u(-\sigma_u, X_u(\sigma_u)) \prod_{u\in A_{\tau}^{(n+1)}} \mathbf u^{1+A_u}(-d_u,X_u(d_u)) \;|\; \H_{n+1} \bigg{)}\\
	=& \prod_{|u|=n+1,u\notin D_{\tau}} \E_x \left( \textbf{1}_{\{\sigma_u<d_u \}} \mathbf u(-\sigma_u, X_u(\sigma_u)) + \textbf{1}_{\{\sigma_u \geq d_u \}} \mathbf u^{1+A_u}(-d_u, X_u(d_u))  \;|\; \H_{n+1} \right)\\
	=& \prod_{|u|=n+1,u\notin D_{\tau}} \E_{X_u(b_u)} \left( \textbf{1}_{\{\sigma_u<d_u \}} \mathbf u(-\sigma_u, X_u(\sigma_u-b_u)) + \textbf{1}_{\{\sigma_u \geq d_u \}} \mathbf u^{1+A_u}(-d_u, X_u(d_u-b_u)) \right)\\
	=& \prod_{|u|=n+1,u\notin D_{\tau}} \E_{X_u(b_u)} \left( \textbf{1}_{\{\sigma_u-b_u<d_u-b_u \}} \mathbf v_u(-(\sigma_u-b_u), X_u(\sigma_u-b_u)) + \right.\\
	&\hspace{3.5cm}\left.\textbf{1}_{\{\sigma_u -b_u \geq d_u-b_u \}} \mathbf v_u^{1+A_u}(-(d_u-b_u), X_u(d_u-b_u)) \right)\\
	=& \prod_{|u|=n+1,u\notin D_{\tau}} \mathbf v_u(0, X_u(b_u)),
	\end{align*}
	where $\mathbf v_u(t,x) = \mathbf u(t-b_u,x)$ and given $b_u$, $\mathbf v_u(t,x)$ satisfies \eqref{KPPeq2} and \eqref{pul_travel}. The last equality follows from Lemma \ref{lemma_prod_mean}, where the only difference is that we substitute the random time $\sigma_u-b_u$ for the deterministic time $t$ in Lemma \ref{lemma_prod_mean}. Putting together the offspring of the same particle it becomes
	\begin{align*}
	&\prod_{|u|=n+1,u\notin D_{\tau}} \mathbf v_u(0, X_u(b_u)) = \prod_{|u|=n+1,u\notin D_{\tau}} \mathbf u(-b_u, X_u(b_u))\\
	= &\prod_{|u|=n, u\notin D_{\tau},u\notin L_{\tau}}\mathbf u^{1+A_u}(-d_u, X_u(d_u))
	= \prod_{u\in A_{\tau}^{(n)}} \mathbf u^{1+A_u}(-d_u, X_u(d_u)).
	\end{align*}
	This shows that $M^{(n)}$ is a $\H_{n+1}$-martingale.
	
	Since $A_{\tau}^{(n)} \rightarrow\emptyset$ as $n\rightarrow\infty$, $M^{(n)}\rightarrow M_x(\nu)$ almost surely. Since $0\leq \mathbf u(t,x) \leq 1$, $M^{(n)}$ is bounded. This yields the $L^1$-convergence of $M^{(n)}$ to $M_x(\nu)$. Therefore,
	\begin{equation*}
	\E_z(M_{x}(v)) = \E_z (M^{(0)}) = \mathbf u(0,z).
	\end{equation*}
	This completes the proof of this lemma.
\end{proof}

Now we turn to the proof of Theorem \ref{thrm_mart_prod}.
\begin{proof}[Proof of Theorem \ref{thrm_mart_prod}]
	Fix $\nu\ge \nu^*$.
	To distinguish the time when a particle hits  different barriers, we let $\sigma_u^x$ to denote the time when $u$ hits the barrier $\Gamma^{(x,\nu)}$. For $y>x$,
	\begin{align*}
	M_{y}(\nu) = \prod_{u\in C(y,\nu)} \mathbf u(-\sigma_u^y, X_u(\sigma_u^y))
	= \prod_{w\in C(x,\nu)} \prod_{u\in C(y,\nu), u>w}  \mathbf u(-\sigma_u^y, X_u(\sigma_u^y)).
	\end{align*}
	Therefore, by the special Markov property of $\{Z_t, t\geq 0\}$,
	we have that
	\begin{align*}
	\E_x (M_{y}(\nu) \;|\; \F_{C(x,\nu)})
	&= \prod_{w\in C(x,\nu)} \E_x\left(\prod_{u\in C(y,\nu), u>w}  \mathbf u(-\sigma_u^y, X_u(\sigma_u^y)) \;|\; \F_{C(x,\nu)}\right)\\
	&= \prod_{w\in C(x,\nu)} \E_{X_w(\sigma_w^x)} \prod_{u\in C(y,\nu), u>w} \mathbf u(-\sigma_u^y, X_u(\sigma_u^y-\sigma_w^x) )\\
	&= \prod_{w\in C(x,\nu)} \mathbf u(-\sigma_w^x, X_u(\sigma_w^x) ) = M_{x}(\nu),
	\end{align*}
	where the second to last equality follows from Lemma \ref{lemma_M_tau} and an argument similar to that
	in the proof of Lemma \ref{lemma_M_tau} by defining $\mathbf v(t,x) = \mathbf u(t-\sigma_w^x,x)$. The proof is complete.
\end{proof}
\subsection{Uniqueness in the supercritical and critical cases}

In this section, we give a probabilistic proof of the uniqueness of the pulsating travelling wave with speed $|\nu|\geq \nu^*$.

Theorem \ref{thrm_asym_super} implies that for $\nu>\nu^*$
\begin{equation}\label{asym_super_log}
-\log \mathbf u\left(\frac{y-x}{\nu}, y\right) \sim \beta e^{-\lambda x}\psi(y,\lambda) \text{ as } x\rightarrow +\infty \text{ uniformly in } y\in [0,1].
\end{equation}
Theorem \ref{thrm_asym_crit} implies that
\begin{equation}\label{asym_criti_log}
-\log \mathbf u\left(\frac{y-x}{\nu^*}, y\right) \sim \beta xe^{-\lambda^* x}\psi(y,\lambda^*) \text{ as } x\rightarrow +\infty \text{ uniformly in } y\in [0,1].
\end{equation}

Recall that $C(x,\nu)$ was defined at the beginning of Section \ref{s:5.3}.
In the spirit of \cite{Ky}, we define
\begin{equation}\label{mart_W_stop}
W_{C(x,\nu)}(\lambda) := \sum_{u\in C(x,\nu)} e^{-\lambda X_u(\sigma_u)-\gamma(\lambda) \sigma_u} \psi(X_u(\sigma_u),\lambda),
\end{equation}
where $\nu = \gamma(\lambda)/\lambda$.
Using arguments similar to those of \cite[Theorem 8]{Ky}, we can obtain the following result, whose proof is omitted.
\begin{prop}\label{thrm_W}
For any $z\in \R$,  $\{W_{C(x,\nu)}(\lambda): x\geq z\}$ is a $\P_z$-martingale with respect to the filtration $\{\F_{C(x,\nu)}: x\geq z \}$, and, as $x\to\infty$,  $W_{C(x,\nu)}(\lambda)$ converges almost surely and in $L^1(\P_z)$ to $W(\lambda,z)$ when $|\lambda|\in [0,\lambda^*)$ and $\mathbf E(L\log^+L)<\infty$.
\end{prop}
Let $Y_x = \sum_{u\in C(x,\nu)} \delta_{\{X_u(\sigma_u) \}}$ where $\{x\}$ is the fractional part of $x$. Then $Y_x$ is a point measure on $[0,1]$.
Notice that
\begin{align*}
W_{C(x,\nu)}(\lambda) &= \sum_{u\in C(x,\nu)} e^{-\lambda (X_u(\sigma_u) + \nu \sigma_u)} \psi(X_u(\sigma_u),\lambda) \\
&= \sum_{u\in C(x,\nu)} e^{-\lambda x} \psi(X_u(\sigma_u),\lambda) = e^{-\lambda x} \langle Y_x, \psi \rangle.
\end{align*}
Thus by Proposition \ref{thrm_W}, we have
\begin{equation}\label{limit-Y-psi}
e^{-\lambda x} \langle Y_x, \psi \rangle \overset{\P_z\mbox{-a.s.}}{\longrightarrow} W(\lambda,z).
\end{equation}

\begin{thrm}\label{uniq-super}
Suppose $|\nu|>\nu^*$ and $\mathbf E(L\log^+L) < \infty$. If $\mathbf u(t,x)$ is a pulsating travelling wave with speed $\nu$,
then there exists $\beta>0$ such that
\begin{equation}
\label{def-u-supercritical2}
\mathbf u(t,x)= \E_x \exp\left\{-\beta e^{\gamma(\lambda)t} W(\lambda,x) \right\},
\end{equation}
where $|\lambda| \in (0,\lambda^*)$ is such that
$\nu = \frac{\gamma(\lambda)}{\lambda}$.
\end{thrm}

\begin{proof}
We assume that $\lambda\geq 0$.  The case $\lambda < 0$ can be analyzed by symmetry.
By Theorem \ref{thrm_mart_prod}, $M_{x}(\nu)$ is a $\P_z$-martingale with respecct to $\{\F_{C(x,\nu)}: x\geq z\}$ with expectation $\mathbf u(0,z)$ and converges almost surely and in $L^1(\P_z)$, where
 $$M_{x}(\nu) = \exp\left\{\sum_{u\in C(x,\nu)} \log \mathbf u(-\sigma_u, X_u(\sigma_u)) \right\}.$$ So there exists a non-negative random variable $Y$ such that
\begin{equation*}
-\sum_{u\in C(x,\nu)} \log \mathbf u(-\sigma_u, X_u(\sigma_u)) \overset{\P_z\mbox{-a.s.}}{\longrightarrow} Y.
\end{equation*}
Note that
\begin{align*}
&X_u(\sigma_u) = \{ X_u(\sigma_u)\} + \lfloor X_u(\sigma_u)\rfloor ,\\
&-\sigma_u = \frac{X_u(\sigma_u)-x}{\nu} = \frac{\{ X_u(\sigma_u)\}-x}{\nu} + \frac{\lfloor X_u(\sigma_u)\rfloor}{\nu}.
\end{align*}
The previous convergence can be written as
\begin{equation*}
\left\langle Y_x, \;- \log \mathbf u\left(\frac{\cdot-x}{\nu}, \cdot\right)\right\rangle \overset{\P_z\mbox{-a.s.}}{\longrightarrow} Y,\quad \mbox{as }x\to\infty.
\end{equation*}
By \eqref{asym_super_log},
$$
\lim_{x\to\infty}\left\langle Y_x, - \log \mathbf u\left(\frac{\cdot-x}{\nu}, \cdot\right)\right\rangle=\beta \lim_{x\to\infty}\langle e^{-\lambda x} Y_x, \psi(\cdot,\lambda)\rangle,
$$
and thus
$Y=\beta W(\lambda,z)$.
By the dominated convergence theorem
\begin{align*}
\mathbf u(0,z) &= \lim_{x\rightarrow\infty} \E_z e^{\left\langle Y_x, \; \log \mathbf u\left(\frac{\cdot-x}{\nu}, \cdot\right)\right\rangle} = \E_z \lim_{x\rightarrow\infty} e^{\left\langle Y_x, \; \log u\left(\frac{\cdot-x}{\nu}, \cdot\right)\right\rangle}\\
&= \E_z e^{-Y} =  \E_z e^{-\beta W(\lambda,z)}.
\end{align*}
\cite[Theorem 1.3(i)]{RSYa} shows that $\E_x  \exp\{-\beta e^{\gamma(\lambda)t} W(\lambda,x)\}$,  as a function of $(t,x)$, is a solution of the following initial value problem:
\begin{align*}
\frac{\partial \mathbf u}{\partial t} = \frac{1}{2} \frac{\partial^2 \mathbf u}{\partial x^2} +
\mathbf g\cdot(\mathbf f(\mathbf u)-\mathbf u),
\quad \mathbf u(0,x) = \E_x e^{-\beta W(\lambda,x)}.
\end{align*}
Therefore, $\mathbf u(t,x)$ and $\E_x  \exp\{-\beta e^{\gamma(\lambda)t} W(\lambda,x)\}$ are solutions of the above initial value problem.
The uniqueness of solutions of  initial value problem implies \eqref{def-u-supercritical2} holds.
\end{proof}

Now we consider the critical case.
Recall that on the space-time half plane $\{(y,t): \, y\in\R, \, t\in\R^+  \}$, the barrier $\Gamma^{(x,\nu)}$ is described by the line $y+\nu t=x$ for $x>0$ and $C(x,\nu)$ is the random collection of particles stopped at the barrier. We have also defined the barrier
 $\mathbf{\Gamma}^{(-x,\lambda)}$ described by $y=h^{-1}(-x-\gamma'(\lambda)t)$ and $\mathbf{C}(-x,\lambda)$ is defined as the random collection of particles hitting this barrier.

Define $\widetilde{C}(z,\nu^*)$ to be the set of particles
that stopped the barrier $\Gamma^{(z,\nu^*)}$ before meeting the barrier
$\mathbf{\Gamma}^{(-x,\lambda^*)}$.
Fix $x>0$ and $y>h^{-1}(-x)$.
Define $Y_z^{(-x,\lambda^*)} = \sum_{u\in \widetilde{C}(z,\nu^*)} \delta_{\{X_u(\sigma_u) \}}$, $z\geq y$.
Consider the sequence $\{V_{\widetilde{C}(z,v^*)}^x, z\geq y\}$, where
\begin{align*}
&V_{\widetilde{C}(z,\nu^*)}^x \\
:=& \sum_{u\in \widetilde{C}(z,\nu^*)} e^{-\gamma(\lambda^*)\sigma_u - \lambda^* X_u(\sigma_u)} \psi(X_u(\sigma_u),\lambda^*) \left(x + \gamma'(\lambda^*)\sigma_u + X_u(\sigma_u) - \frac{\psi_{\lambda}(X_u(\sigma_u),\lambda^*)}{\psi(X_u(\sigma_u),\lambda^*)} \right)\\
=& \sum_{u\in \widetilde{C}(z,\nu^*)} e^{- \lambda^* z} \psi(X_u(\sigma_u),\lambda^*) \left(x + z - \frac{\psi_{\lambda}(X_u(\sigma_u),\lambda^*)}{\psi(X_u(\sigma_u),\lambda^*)} \right)\\
=& \left\langle Y_z^{(-x,\lambda^*)}, \, e^{- \lambda^* z} \psi(\cdot,\lambda^*) \left(x + z - \frac{\psi_{\lambda}(\cdot,\lambda^*)}{\psi(\cdot,\lambda^*)} \right) \right\rangle,\quad
z\geq y.
\end{align*}
Using similar arguments as in \cite[Theorem 15]{Ky}, we have following proposition.

\begin{prop}\label{lemma_martV}
 	Let $\{\mathcal{F}_{\widetilde{C}(z,\nu^*)}: z\geq y \}$ be the natural filtration describing everything in the truncated branching tree up to the barrier $\Gamma^{(z,\nu^*)}$. If $\mathbf E(L(\log^+L)^2)<+\infty$, then $\{V_{\widetilde{C}(z,\nu^*)}^x, z\geq y\}$ is a $\P_y$-martingale with respect to $\{\F_{\widetilde{C}(z,\nu^*)}, z\geq y\}$, and $V_{\widetilde{C}(z,\nu^*)}^x$  converges $\P_y$-almost surely and in $L^1(\P_y)$ to $V^x(\lambda^*)$ as $z\to\infty$.
\end{prop}
\begin{proof} For $t>0$, let
\begin{equation*}
\widetilde{C}_t(z,\nu^*) = \{u\in \widetilde{C}(z,\nu^*): \sigma_u \leq t \}
\end{equation*}
and
\begin{equation*}
\widetilde A_t(z,\nu^*)
= \{u\in \widetilde{N}_t^x: \, v\notin \widetilde{C}_t(z,\nu^*), \, \forall v\leq u \}.
\end{equation*}
Define
\begin{align*}
V_{t\wedge\widetilde{C}(z,\nu^*)}^x := &\sum_{u\in\widetilde A_t(z,\nu^*)} e^{-\gamma(\lambda^*)t - \lambda^* X_u(t)} \psi(X_u(t),\lambda^*) \left(x + \gamma'(\lambda^*)t + X_u(t) - \frac{\psi_{\lambda}(X_u(t),\lambda^*)}{\psi(X_u(t),\lambda^*)} \right) \\
&\,+ \sum_{u\in \widetilde{C}_t(z,\nu^*)} e^{- \lambda^* z} \psi(X_u(\sigma_u),\lambda^*) \left(x + z - \frac{\psi_{\lambda}(X_u(\sigma_u),\lambda^*)}{\psi(X_u(\sigma_u),\lambda^*)} \right).
\end{align*}
A straightforward calculation, similar to the proof of
\cite[Lemma 2.16]{RSYa},
shows that
\begin{equation}\label{mart_V_stopping}
\E_y (V_t^x(\lambda^*) | \mathcal{F}_{\widetilde{C}(z,\nu^*)})  = V_{t\wedge\widetilde{C}(z,\nu^*)}^x.
\end{equation}
Since $\lim_{t\uparrow\infty} |\widetilde A_t(z,\nu^*)| = 0$ and $\lim_{t\uparrow\infty} \widetilde{C}_t(z,\nu^*) = \widetilde{C}(z,\nu^*)$. Letting $t\to\infty$ in \eqref{mart_V_stopping}, we have
\begin{equation*}
\lim_{t\uparrow\infty} \E_y (V_t^x(\lambda^*) | \mathcal{F}_{\widetilde{C}(z,\nu^*)}) =  V_{\widetilde{C}(z,\nu^*)}^x.
\end{equation*}
By \cite[Theorem 4.2]{RSYa},
$V_t^x(\lambda^*)$ converges  to $V^x(\lambda^*)$  in $L^1(\P_y)$ as $t\to\infty$.
Thus $\E_y (V_t^x(\lambda^*) | \mathcal{F}_{\widetilde{C}(z,\nu^*)})$ converges to $\E_y (V^x(\lambda^*) | \mathcal{F}_{\widetilde{C}(z,\nu^*)})$ in $L^1(\P_y)$.
So $\E_y (V^x(\lambda^*) | \mathcal{F}_{\widetilde{C}(z,\nu^*)}) = V_{\widetilde{C}(z,\nu^*)}^x$.

Letting $z\to\infty$  in \eqref{mart_V_stopping}, we get $\E_y (V_t^x(\lambda^*) | \mathcal{F}_{\infty}) = V_t^x(\lambda^*)$ where $\F_{\infty} = \sigma(\cup_{z\geq y} \mathcal{F}_{\widetilde{C}(z,\nu^*)} )$. This implies that $V^x(\lambda^*)$ is $\F_{\infty}$-measurable. Hence
\begin{equation*}
 V_{\widetilde{C}(z,\nu^*)}^x = \E_y \left(V^x(\lambda^*) | \mathcal{F}_{\widetilde{C}(z,\nu^*)}\right) \overset{L^1(\P_y)/a.s.}{\longrightarrow} \E_y \left(V^x(\lambda^*) \Big| \F_{\infty}\right) = V^x(\lambda^*).
\end{equation*}
This completes the proof.
\end{proof}

\begin{thrm}\label{uniq-critical} Suppose
$\mathbf E(L(\log^+L)^2) < \infty$. If $\mathbf u(t,x)$ is a pulsating travelling wave with speed $\nu^*$,
then there exists $\beta>0$ such that
\begin{equation}\label{def-u-critical}
\mathbf u(t,x)= \E_x \exp\left\{-\beta e^{\gamma(\lambda^*)t} \partial W(\lambda^*,x) \right\}.
\end{equation}
\end{thrm}

\begin{proof}

From Proposition \ref{lemma_martV}, we have
\begin{equation*}
\lim_{z\rightarrow\infty} V_{\widetilde{C}(z,\nu^*)}^x =
\lim_{z\rightarrow\infty}
\left\langle Y_z^{(-x,\lambda^*)}, \, e^{- \lambda^* z} \psi(\cdot,\lambda^*)\left(x + z - \frac{\psi_{\lambda}(\cdot,\lambda^*)}{\psi(\cdot,\lambda^*)} \right) \right\rangle = V^x(\lambda^*), \quad \P_y\text{-a.s.}
\end{equation*}
Notice that for fixed $x\geq 0$,
\begin{equation*}
\left(x + z - \frac{\psi_{\lambda}(w,\lambda^*)}{\psi(w,\lambda^*)} \right) / z \rightarrow 1 \text{ as $z\rightarrow\infty$ uniformly in $w\in [0,1]$}.
\end{equation*}
Therefore,
\begin{equation*}
\lim_{z\rightarrow\infty}
\left\langle Y_z^{(-x,\lambda^*)}, \,  z e^{- \lambda^* z} \psi(\cdot,\lambda^*)  \right\rangle = V^x(\lambda^*), \quad \P_y\text{-a.s.}
\end{equation*}
Recall that $\gamma^{(-x,\lambda^*)}$ is the event that the BBMPE remains entirely to the right of
$\mathbf{\Gamma}^{(-x,\lambda^*)}$.
Note that on the event $\gamma^{(-x,\lambda^*)}$, $V^x(\lambda^*) = \partial W(\lambda^*,y)$ $\P_y$-almost surely and $Y_z^{(-x,\lambda^*)} = Y_z$ where $Y_z = \sum_{u\in C(z,\nu^*)} \delta_{\{X_u(\sigma_u) \}}$. Thus it follows that under $\P_y$,
\begin{equation*}
\lim_{z\rightarrow\infty}
\langle Y_z, \,  z e^{- \lambda^* z} \psi(\cdot,\lambda^*)  \rangle = \partial W(\lambda^*,y) \quad \text{on } \gamma^{(-x,\lambda^*)}.
\end{equation*}
Using the fact $\P(\gamma^{(-x,\lambda^*)})\rightarrow 1$ as $x\rightarrow\infty$, we have
\begin{equation*}
\lim_{z\rightarrow\infty}
\left\langle Y_z, \,  z e^{- \lambda^* z} \psi(\cdot,\lambda^*)  \right\rangle = \partial W(\lambda^*,y) \quad \P_y\text{-a.s.}
\end{equation*}
So by the dominated convergence theorem and the asymptotic behavior \eqref{asym_criti_log},
\begin{align*}
\mathbf u(0,y) &= \lim_{z\rightarrow\infty} \E_y e^{\left\langle Y_z, \; \log \mathbf u\left(\frac{\cdot-z}{\nu^*}, \cdot\right)\right\rangle} = \E_y \lim_{z\rightarrow\infty} e^{-\left\langle Y_z, \; -\log \mathbf u\left(\frac{\cdot-z}{\nu^*}, \cdot\right)\right\rangle}\\
&= \E_y \lim_{z\rightarrow\infty} e^{-\left\langle Y_z, \; \beta ze^{-\lambda z}\psi(\cdot,\lambda^*)\right\rangle} = \E_y e^{-\beta \partial W(\lambda^*,y)}.
\end{align*}
\cite[Theorem 1.3(ii)]{RSYa} shows that $ \E_x \exp\left\{-\beta e^{\gamma(\lambda^*)t} \partial W(\lambda^*,x) \right\}$,  as a function of $(t,x)$, is a solution of the following initial value problem:
\begin{align*}
\frac{\partial \mathbf u}{\partial t} = \frac{1}{2} \frac{\partial^2 \mathbf u}{\partial x^2} + \mathbf g\cdot (\mathbf f(\mathbf u)-\mathbf u),\quad \mathbf u(0,x) = \E_x e^{-\beta \partial W(\lambda^*,x)}.
\end{align*}
Therefore, $\mathbf u(t,x)$ and $ \E_x \exp\left\{-\beta e^{\gamma(\lambda^*)t} \partial W(\lambda^*,x) \right\}$ are solutions of the above initial value problem.
The uniqueness of solutions of initial value problem implies
\eqref{def-u-critical} holds.
\end{proof}

\begin{proof}[Proof of Theorem \ref{thrm3}]
Combining Theorem \ref{uniq-super}, Theorem \ref{uniq-critical} and \cite[Theorem 1.3]{RSYa}, we have Theorem \ref{thrm3}.
\end{proof}

\section{Appendix}

\subsection{Appendix B}

\begin{lemma}\label{local mart}
Let $\{f(B_t), t\geq 0\}$ be defined by \eqref{def-f-St}. For any $0<s<t$,
\begin{align}
\Pi_x^{(y,\lambda^*)}\left[f(B_{t\wedge\tau_z})| \F_{s}\right]= f(B_{s\wedge\tau_z}),
\end{align}
where $\tau_{z}$ is defined by \eqref{tau-z},
\end{lemma}
\begin{proof}
For any  $s>0$, we use $\theta_s$ to denote the  shift operator.
First note that
\begin{align}\label{I+II}
&\Pi_x^{(y,\lambda^*)}\left(f(B_{t\wedge\tau_z}) | \F_{s}\right)
\\
= &\Pi_x^{(y,\lambda^*)}\left[\widehat {\mathbf w}(-t\wedge\tau_z ,B_{t\wedge\tau_z},y) e^{-\int_0^{t\wedge\tau_z} \mathbf g(B_r)\mathbf w(-r,B_r)\mathrm{d}r} \mathbf{1}_{s\geq\tau_z}\Big| \F_{s}\right]\notag\\
&+\Pi_x^{(y,\lambda^*)}\left[\widehat {\mathbf w}(-t\wedge\tau_z ,B_{t\wedge\tau_z},y) e^{-\int_0^{t\wedge\tau_z} \mathbf g(B_r)\mathbf w(-r,B_r)\mathrm{d}r} \mathbf{1}_{s<\tau_z}\Big| \F_{s}\right]\notag\\
=:&I+II.\notag
\end{align}
For $I$ we have
\begin{align}\label{for-I}
I=\widehat {\mathbf w}(-\tau_z ,B_{\tau_z},y) e^{-\int_0^{\tau_z} \mathbf g(B_r)\mathbf w(-r,B_r)\mathrm{d}r} \mathbf{1}_{s\geq\tau_z}.
\end{align}
For $II$ we will prove that
\begin{align}\label{for-II}
II=\widehat {\mathbf w}(-s ,B_{s},y) e^{-\int_0^{s} \mathbf g(B_r)\mathbf w(-r,B_r)\mathrm{d}r} \mathbf{1}_{s<\tau_z},
\end{align}
which is equivalent to
\begin{align}\label{for-II2}
&\Pi_x \left( \frac{\Lambda_{t\wedge\tau_z}^{(y,\lambda^*)}} {\Lambda_{0}^{(y,\lambda^*)}} \widehat {\mathbf w}(-t\wedge\tau_z ,B_{t\wedge\tau_z},y) e^{-\int_0^{t\wedge\tau_z} \mathbf g(B_r)\mathbf w(-r,B_r)\mathrm{d}r} \mathbf{1}_{\{\Lambda_{t\wedge\tau_z}^{(y,\lambda^*)}>0 \}} \mathbf{1}_{s<\tau_z}\Big| \F_{s}\right)\\
=&\frac{\Lambda_{s}^{(y,\lambda^*)}}{\Lambda_{0}^{(y,\lambda^*)}}\widehat {\mathbf w}(-\tau_z ,B_{s},y) e^{-\int_0^{s} \mathbf g(B_r)\mathbf w(-r,B_r)\mathrm{d}r} \mathbf{1}_{s<\tau_z}.
\end{align}
Recall that
$$\tau_{z} = \tau_z(y) = \inf\{t\geq 0: y+\gamma'(\lambda^*)t+h(B_t) \leq z \},$$
and $\tau_z(y+\nu^*s) = \tau_{z-\nu^*s}$. For $0<s<t$, we have on $\{s<\tau_z\}$,
$$
t\wedge\tau_z = s + (t-s)\wedge\tau_{z-\nu^*s}\circ \theta_{s},$$
and
$$
B_{t\wedge\tau_z} = B_{(t-s)\wedge\tau_{z-\nu^*s}}\circ\theta_{s}.
$$
Using the Markov property of $\{B_t,t\geq 0\}$ and the fact that $\{\Lambda_{t\wedge\tau_z}^{(y,\lambda^*)}>0\}$, we have
\begin{align}
&\mbox{left side of \eqref{for-II2}}
\\=&
\Pi_x \left(  (\Lambda_{0}^{(y,\lambda^*)})^{-1}e^{-\gamma(\lambda^*)t\wedge\tau_z -\lambda^* B_{t\wedge\tau_z} + \int_0^{t\wedge\tau_z}\mathbf g(B_r)\mathrm{d}r} \psi(B_{t\wedge\tau_z},\lambda^*)(y+\gamma'(\lambda^*)t\wedge\tau_z + h(B_{t\wedge\tau_z})) \right.
\\&
\quad \left. \widehat {\mathbf w}(-t\wedge\tau_z ,B_{t\wedge\tau_z},y) e^{-\int_0^{t\wedge\tau_z} \mathbf g(B_r)\mathbf w(-r,B_r)\mathrm{d}r}\mathbf{1}_{s<\tau_z} \big| \F_{s} \right)
\\=  &
 \mathbf{1}_{s<\tau_z}(\Lambda_{0}^{(y,\lambda^*)})^{-1} e^{-\gamma(\lambda^*)s  + \int_0^{s}\mathbf g(B_r)(1-\mathbf w(-r,B_r))\mathrm{d}r} \\
&\times\Pi_x \Big{[}  e^{-\gamma(\lambda^*)((t-s)\wedge\tau_{z-\nu^*s})-\lambda^* B_{(t-s)\wedge\tau_{z-\nu^*s}} + \int_0^{(t-s)\wedge\tau_{z-\nu^*s}}\mathbf g(B_r)(1-\mathbf w(-s-r,B_r))\mathrm{d}r}   \\
& \times \psi(B_{(t-s)\wedge\tau_{z-\nu^*s}},\lambda^*)(y+ \gamma'(\lambda^*)(s+(t-s)\wedge\tau_{z-\nu^*s})+ h(B_{(t-s)\wedge\tau_{z-\nu^*s}}))\\
&\times \widehat {\mathbf w}(-s -(t-s)\wedge\tau_{z-\nu^*s}, B_{(t-s)\wedge\tau_{z-\nu^*s}},y) \circ \theta_{s} \Big{|}\mathcal{F}_{s} \Big{]}
\\
= &\mathbf{1}_{s<\tau_z}(\Lambda_{0}^{(y,\lambda^*)})^{-1}e^{-\gamma(\lambda^*)s  + \int_0^{s}\mathbf g(B_r)(1-\mathbf w(-r,B_r))\mathrm{d}r} \\
&\times\Pi_{B_{s}}\Big{[}  e^{-\gamma(\lambda^*)((t-s)\wedge\tau_{z-\nu^*s}) -\lambda^* B_{(t-s)\wedge\tau_{z-\nu^*s}} + \int_0^{(t-s)\wedge\tau_{z-\nu^*s}}\mathbf g(B_r) (1-\mathbf w(-s-r,B_r))\mathrm{d}r}  \\
&\times \psi(B_{(t-s)\wedge\tau_{z-\nu^*s}},\lambda^*) (y+ \gamma'(\lambda^*)s + \gamma'(\lambda^*)((t-s)\wedge\tau_{z-\nu^*s})+ h(B_{(t-s)\wedge\tau_{z-\nu^*s}}))\\
&\times \widehat {\mathbf w}(-s-(t-s)\wedge\tau_{z-\nu^*s} ,B_{(t-s)\wedge\tau_{z-\nu^*s}},y) \Big{]}.
\end{align}
By \eqref{meas_change}, we have
\begin{align}
&\mbox{the left side of \eqref{for-II2}}\\
=& \Pi_{B_{s}}^{(y+\nu^*s,\lambda^*)}\left[   e^{-\int_0^{(t-s)\wedge\tau_{z-\nu^*s}}\mathbf g(B_r)\mathbf w(-s-r,B_r)\mathrm{d}r}  \widehat {\mathbf w}(-s-(t-s)\wedge\tau_{z-\nu^*s} ,B_{(t-s)\wedge\tau_{z-\nu^*s}},y) \right]\\
&\times \mathbf{1}_{s<\tau_z}(\Lambda_{0}^{(y,\lambda^*)})^{-1}  e^{-\gamma(\lambda^*)s + \int_0^{s}\mathbf g(B_r)(1-\mathbf w(-r,B_r))\mathrm{d}r} e^{-\lambda^* B_{s}}  \psi(B_{s},\lambda^*) (y+\nu^*s+h(B_{s}))\\
=& \Pi_{B_{s}}^{(y+\nu^*s,\lambda^*)}\left[   e^{-\int_0^{(t-s)\wedge\tau_{z-\nu^*s}}\mathbf g(B_r)\mathbf w(-s-r,B_r)\mathrm{d}r}  \widehat {\mathbf w}(-s-(t-s)\wedge\tau_{z-\nu^*s} ,B_{(t-s)\wedge\tau_{z-\nu^*s}},y) \right]\\
&\times \mathbf{1}_{s<\tau_z} (\Lambda_{0}^{(y,\lambda^*)})^{-1}\Lambda_{s}^{(y,\lambda^*)} e^{-\int_0^{s}\mathbf g(B_r)\mathbf w(-r,B_r)\mathrm{d}r} \notag  \\
=& \Pi_{B_{s}}^{(y+\nu^*s,\lambda^*)}\left[   e^{-\int_0^{(t-s)\wedge\tau_{z-\nu^*s}}\mathbf g(B_r)\mathbf w(-s-r,B_r)\mathrm{d}r}  \widehat {\mathbf w}(-s-(t-s)\wedge\tau_{z-\nu^*s} ,B_{(t-s)\wedge\tau_{z-\nu^*s}},y) \right]\\
&\times \frac{\Lambda_{s}^{(y,\lambda^*)}}{\Lambda_{0}^{(y,\lambda^*)}} e^{-\int_0^{s\wedge\tau_z}\mathbf g(B_r)\mathbf w(-r,B_r)\mathrm{d}r}\mathbf{1}_{s<\tau_z}\\
= &\frac{\Lambda_{s}^{(y,\lambda^*)}}{\Lambda_{0}^{(y,\lambda^*)}}\widehat {\mathbf w}(-s ,B_{s},y)   e^{-\int_0^{s\wedge\tau_z}\mathbf g(B_r)\mathbf w(-r,B_r)\mathrm{d}r}\mathbf{1}_{s<\tau_z}
=\mbox{the right side of \eqref{for-II2}},
\end{align}
where in the last equality we used \eqref{v_feynman_crit}  with $T$ replaced by $-s$, $x$ replaced by $B_{s}$, $y$ replaced by $y+\nu^*s$,
$t$ replaced by $t-s$, and $z$ replaced by $z-\nu^*s$.
Hence \eqref{for-II} holds.
Combining \eqref{I+II}, \eqref{for-I} and \eqref{for-II}, we obtain
\begin{align}
& \Pi_x^{(y,\lambda^*)}\left[\widehat {\mathbf w}(-t\wedge\tau, B_{t\wedge\tau_z},y) e^{-\int_0^{t\wedge\tau_z} \mathbf g(B_r)\mathbf w(-r,B_r)\mathrm{d}r}\big|\F_s\right]\\
=& \widehat {\mathbf w}(-s\wedge\tau, B_{s\wedge\tau}, y) e^{-\int_0^{s\wedge\tau} \mathbf g(B_r)\mathbf w(-r,B_r)\mathrm{d}r}
=f(B_{s\wedge\tau}).
\end{align}
\end{proof}



\begin{thebibliography}{99}

\bibitem{Athreya} K. B. Athreya and P. E. Ney. \emph{Branching processes.} Springer, Berlin Heidelberg, New York, 1972.


\bibitem{Bramson78} M. Bramson. Maximal displacement of branching Brownian motion. \emph{Common. Pure Appl. Math.} {\bf31} (1978) 531-581.


\bibitem{Bramson83} M. Bramson. Convergence of solutions to the Kolmogorov equation to travelling waves. \emph{Mem. Amer. Math. Soc.} {\bf44} (1983) iv+190 pp.


\bibitem{CR90} B. Chauvin and A. Rouault. Supercritical branching Brownain motion and K-P-P equation in the critical speed-area \emph{Math. Nachr.} {\bf 149} (1990) 41-59.

\bibitem {Chauvin91} B. Chauvin. Product martingales and stopping lines for branching Brownian motion. \emph{Ann. Probab.} {\bf 30} (1991) 1195-1205.


\bibitem{Fish}  R. A. Fisher. The wave of advance of advantageous genes. \emph{Ann. Eugenics} {\bf7} (1937)
355-369.

\bibitem{Harris99} S. C. Harris. Travelling waves for the F-K-P-P equation via probabilistic arguments. \emph{Proc. Roy. Soc. Edinburgh Sect. A} {\bf129} (1999) 503-517.	


\bibitem{H08}
F. Hamel. Qualitative properties of monostable pulsating fronts: exponential decay and monotonicity.
\emph{J. Math. Pures Appl.} {\bf 89} (2008) 355-399.

\bibitem{HR} F. Hamel, and L. Roques. Uniqueness and stability properties of monostable pulsating fronts.
\emph{J. Eur. Math. Soc.} {\bf 13} (2011) 345-390.

\bibitem{HNRR} F. Hamel, J. Nolen, J.-M. Roquejoffre and L. Ryzhik. The logarithmic delay of KPP fronts in a periodic medium. \emph{J. Eur. Math. Soc. (JEMS)} {\bf18} (2016) 465-505.


\bibitem{KPP} A. Kolmogorov, I. Petrovskii and N. Piskounov. \'{E}tude de I'\'{e}quation de la diffusion avec croissance de la quantit\'{e} de la mati\`{e}re at son application a un probl\`{e}m biologique. \emph{Moscow Univ.Math. Bull.} {\bf1} (1937) 1-25.
	

\bibitem{Ky} A. E. Kyprianou. Travelling wave solution to the K-P-P equation: Alternatives to Simon Harris' probabilistic analysis. \emph{Ann. Inst. Henri Poincar\'{e} Probab. Stat.}
{\bf40} (2004) 53-72.


\bibitem{LTZ}  E. Lubetzky, C. Thornett, and O. Zeitouni.
Maximum of branching Brownian motion in a periodic environment.
\emph{Ann. Inst. Henri Poincar\'{e} Probab. Stat.}, to appear.


\bibitem{Mc} H. P. McKean. Application of Brownian motion to the equation of Kolmogorov-Petrovskii-Piskunov. \emph{Comm. Pure Appl. Math.} {\bf28} (1975) 323-331.


\bibitem{RSYa} Y.-X. Ren, R. Song, and F. Yang. Branching Brownian motion in a periodic environment and existence of pulsating travelling waves.




\end{thebibliography}
\end{document}